\documentclass[12pt,MSc]{muthesis_2023}

\usepackage{amsmath}
\usepackage{amsfonts}
\usepackage{amssymb}
\usepackage{lipsum}
\usepackage[utf8]{inputenc}
\usepackage{graphicx}
\usepackage{tikz}
\usepackage{float}

\usepackage{amsthm}
\usepackage{xpatch}
\xpatchcmd\swappedhead{~}{.~}{}{}
\swapnumbers 

\theoremstyle{plain}
\newtheorem{FactCounter}{dummy}[section]
\newtheorem{Theorem}[FactCounter]{Theorem} %
\newtheorem{Proposition}[FactCounter]{Proposition} %
\newtheorem{Lemma}[FactCounter]{Lemma} %
\newtheorem{Corollary}[FactCounter]{Corollary} %
\theoremstyle{definition}
\newtheorem{Definition}[FactCounter]{Definition} %
\theoremstyle{remark}
\newtheorem{Remark}[FactCounter]{Remark} %
\newtheorem{Example}[FactCounter]{Example} %
\DeclareMathOperator{\tr}{tr}
\DeclareMathOperator{\op}{op}

\DeclareMathOperator{\End}{End}
\DeclareMathOperator{\Rad}{rad}
\DeclareMathOperator{\rad}{rad}

\DeclareMathOperator{\im}{Im}
\DeclareMathOperator{\Hom}{Hom}
\DeclareMathOperator{\Soc}{soc}
\DeclareMathOperator{\soc}{soc}
\DeclareMathOperator{\Ker}{Ker}
\DeclareMathOperator{\Stab}{Stab}

\DeclareMathOperator{\GL}{GL}
\DeclareMathOperator{\ord}{ord}
\DeclareMathOperator{\Inf}{Inf}
 
\newcommand{\ind}{\mathord\uparrow} 
\DeclareMathOperator{\sign}{sign}
\DeclareMathOperator{\codim}{codim}
\DeclareMathOperator{\modd}{mod-}
\DeclareMathOperator{\charr}{char}

\usepackage{cite}

\usepackage{imakeidx}
\makeindex

\usepackage{listings}
\usepackage{hyperref}

\begin{document}
\title{Decompositions of Group Algebras as a Direct Sum of Projective Indecomposable Modules and of Blocks in Positive Characteristic}

\author{Eun H. Park}

\school{Mathematics}
\faculty{Science and Engineering}

\beforeabstract

The dissertation focuses on decomposing a group algebra $kG$ over a field of positive characteristic into a direct sum of projective indecomposable modules. Such a decomposition is obtained using Theorem~\ref{regmodAdecompassumofprojcovers} together with the Artin--Wedderburn Theorem. As the Artin--Wedderburn Theorem is for semisimple modules, it can be used for the quotient algebra $A/\rad A$, known as the head, which is semisimple. The main goal of the dissertation is to explicitly decompose given group algebras as a direct sum of their projective indecomposable modules.  

To achieve this, we determine the radical series of each projective indecomposable module of the given group algebras. For a group algebra over characteristic $p$, each projective indecomposable module has a simple head that is isomorphic to its socle. Projective covers and injective envelopes are used to construct these modules. A cyclic group algebra is uniserial, and a $p$-group algebra over characteristic $p$ is itself a projective indecomposable module. Using these properties, we explicitly find all projective indecomposable modules for the following group algebras over characteristic $2$: the Klein four-group, the alternating group $A_4$, and the alternating group $A_5$. Their relationships play an important role in this process.  

Since $p$-group algebras have trivial head and trivial socle, the Klein four-group algebra has a corresponding radical series. Its decomposition into a direct sum of projective indecomposable modules is described explicitly, and the Cartan matrix of a group algebra is obtained by calculating the multiplicities of simples in its projective indecomposable modules.  

The topic is then extended slightly by considering the unique decomposition of a group algebra into a direct sum of particular modules known as \emph{blocks}. For $kA_4$, the primitive orthogonal idempotents are calculated, and since $kA_4$ has one block, it is equal to its block decomposition. For $kA_5$, we show that there are two blocks, determined by checking the nonzero entries in its Cartan matrix.

\afterabstract

\prefacesection{Acknowledgements}

First and foremost, I would like to express my sincere gratitude to Prof. Eaton for his invaluable guidance and support throughout my dissertation project. Despite time constraints, he helped me understand the topics more deeply and was always welcoming whenever I reached out with questions.

I am also grateful to my colleagues Luca, Amy, Vinnie, and Michael for their encouragement and for sharing discussions on degree progress, academic life, and the dissertation journey. Their warm advice motivated me greatly. I would like to thank Janaid for offering career advice and support during the course, and Alina for her friendship throughout my degree.

I would also like to acknowledge my academic advisor, Dr. Vahagn. Even though I went to his office very often to discuss my career path, he was always welcoming me all the time and gave me very thoughtful advice. I was able to achieve lots of things during the degree mostly by his help.

As one of academic representatives, I would like to thank all the master's pure mathematics and mathematical logic students that we tried so hard to achieve this level and we finally succeed throughout the journey.

I was really honoured to take Galois Theory and Category Theory courses taught by Prof. Kessar and Dr. Gamblino. I learned not only academic contents, but also educating perspective. The learning experience was invaluable. Even though I went to their offices a lot during the courses to ask questions, they are so welcoming. Especially because this topic is relevant to non-commutative algebra, all of Dr. Cesar's contents such as additional readings, lecture video recordings, additional exercises, and his lectures were invaluable resources for developing the background needed for this project. I would also like to thank him for giving me valuable opportunities to engage with the wider mathematical community.

Finally, and most importantly, I am deeply thankful to my family, my father, mom and my sister for their constant love, encouragement, and support throughout the journey.

\afterpreface


\chapter{Introduction}
The dissertation is under the general assumption that any algebra $A$ is finitely generated. When $A$ if finitely generated, $A$ is both Artinian and Noetherian. Some theorems in the dissertation assume that $A$ is Artinian, but it can be assumed as true when $A$ is finitely generated.

In Chapter 2, the concept of regular module, group algebras, projective module, idempotent, and radical is introduced as preliminaries. We show that projective indecomposable modules have one-to-one correspondence with simples. To be specific, projective indecomposable modules are the projective covers of simples. Artin-Wedderburn theorem gives a decomposition of a semisimple module into a direct sum of matrix algebras. By this theorem and using the fact that the head of a regular module is semisimple, a regular module can be decomposed into a direct sum of projective covers of simples with the multiplicity of simple in its head. We proceed this further, but not the main content of the topic, by using the fact that a unique ring decomposition as the direct sum of indecomposable two-sided ideals corresponds with the unique expression of the identity as the sum of primitive orthogonal central idempotents. The direct summand of the unique ring decomposition is called a block. We also want to decompose a regular module as the direct sum of blocks. This chapter gives theoretic justification on such decompositions of a regular module.

To decompose a group algebra as a direct sum of projective indecomposable modules, it starts with investigating how the radical series of projective indecomposable module is structured. The projective cover of the module where its head is simple is used when we want to find a quotient of a projective indecomposable module. Similarly, the dual concept of projective cover, the injective envelope of a module, where its socle is simple, is useful when we want to find a submodule of a projective indecomposable module. In Chapter 3, the radical series structure of a projective indecomposable is the main focus. This chapter is mainly about the dualities, such as projective cover and injective envelope, and radical and socle. A projective indecomposable module of a group algebra over characteristic $p$ has a property that its head is isomorphic to its socle. It means that the top radical layer and the bottom radical layer are the same simple module. The multiplicity of a simple in a module is also introduced in this chapter, in that it is useful to find a submodule or a quotient of the module as it is related to the dimension of a homomorphism space between a projective indecomposable module and a module. So the multiplicity of a simple in a module can be useful to find the radical structure of a module. When finding submodules and quotients of a module, we investigate corresponding homomorphism spaces. 

In Chapter 4-6, we decompose given group algebras: the Klein four-group, the alternating four-group, and the alternating five group over characteristic 2 as the direct sum of projective indecomposable modules. Before that, we see a simple case, when a group is cyclic. One theorem shows that its radical series is uniserial. For $p$-Group algebras, they have a unique simple module, that is, the trivial module. By using this fact, a $p$-Group algebras is itself the projective indecomposable module, so investigating the projective indecomposable module directly gives $p$-Group algebra structure. A $p$-Group algebra has the trivial head and the trivial socle. By using these facts, we derive the structure of the Klein four-group algebra. Simple $kA_4$-modules are obtained from simple $kC_3$-modules. Given indecomposable modules, we will find the projective covers and the injective envelopes of the indecomposable modules to find the projective indecomposable modules. Then the decomposition into projective indecomposable modules is explicitly described as a result.
Additionally, the corresponding primitive orthogonal idempotents are calculated for $kA_4$. $kA_4$ has one block, so $kA_4$  is itself the decomposition into a block.
To find simple $kA_5$-modules, we prove the number of simples is the number of $p$-conjugacy classes of $A_5$. $kA_5$ has four simples, so has four indecomposable modules. We find indecomposable $kA_5$ modules by induction on simple $kA_4$-modules. In the context, we consider $A_4$ as a subgroup of $A_5$. Induced $kA_5$-modules are obtained by considering the dimensions of homomorphism spaces between $kA_4$ simples and the restricted $kA_5$-simples and using Frobenius reciprocity. Like we did in chapter 5, we will find the projective covers and the injective envelops for the indecomposable modules and eventually find some projective indecomposable modules. For one projective indecomposable module, we need to consider homomorphism spaces between the module and another projective indecomposable module to find submodules and quotients of the module. By using projectivity on the induced projective indecomposable module from $A_4$, we can see how the induced module is structured. We can eventually decompose all the given group algebras into a direct sum of projective indecomposable modules. We obtain that $kA_5$ has two blocks by using the relation between the multiplicity of a simple in a projective indecomposable module and a block.

Through this dissertation, we not only investigates given group algebras, but we also gain insight into the general structure of group algebras. The dissertation is broadly based on \cite{C.W.Eaton} and \cite{Webb_2016}.

\chapter{Decomposition of A Regular Module}
\chaptermark{Decomposition of A Regular Module} 

A general assumption of any algebra $A$ or ring $R$ is that it is finitely generated. In non-commutative algebra, if an algebra $A$ is finitely generated, then it is Noetherian\index{Noetherian} and Artinian\index{Artinian}. Another general assumption in this context is that a group $G$ is finite, a field $k$ is algebraically closed (i.e., $\bar{k}=k$), and its characteristic is positive, let say, $p$.

\section{A Regular Module}

As the main object of this dissertation is the regular module, especially the group algebra, we begin by defining the basic concepts.
Let $R$ be a ring. A \textit{regular module}\index{regular module} ${}_RR$ is a left $R$-module acted by $R$ itself where the action is the multiplication in $R$. If the regular module ${}_RR$ seems clear in the context, it is simply denoted as $R$.
A \textit{free}\index{free} module is a module with a basis.
For a ring $R$ and an $R$-module $M$, the subset $E$ of $M$ is called a \textit{basis} for $M$ if it spans $M$ and is linearly independent. \index{group algebra}The \textit{group algebra} $RG$ of $G$ over $R$ is a free $R$-module with the elements of $G$ as an $R$-basis. The multiplication of $RG$ given on the basis elements is extended to arbitrary elements using bilinearity of the operation. An element of $RG$ is of the form $\sum_{g\in G} a_g g$ where $a_g\in R$. The multiplication is defined as $\big(\sum_{g\in G}a_g g\big)\cdot \left(\sum_{h\in G}b_h h\right)=\sum_{k\in G}\big(\sum_{gh=k}a_gb_h\big)k.$ The dimension of $RG$ is $|G|$. For example, the dimension of $kA_4$ and $kA_5$ are $12$ and $60$, respectively, as $|A_4|=12$ and $|A_5|=60$.

\subsection{A Projective Module}
The question is that what types of the direct summands would be when decomposing a regular module into a direct sum of submodules. The direct summands can have a property called projective. Let $R$ be a ring. An $R$-module $P$ is called a \textit{projective}\index{projective} $R$-module if it is a direct summand of a free $R$-module. The definition directly proves that a direct sum of projective modules and a direct summand of a projective module are projective. A projective indecomposable module is often abbreviated as PIM. A projective module has equivalent statements, one with the term split epimorphism.
\index{split epimorphism}
An $R$-module homomorphism $f:M \rightarrow N$ is called \textit{split epimorphism} if there exists an $R$-module homomorphism $g: N\rightarrow M$ such that $fg=1_N$, the identity map on $N$. If a homomorphism is a split epimorphism, then it is an epimorphism, by $x=f(g(x))\in \text{Im}(f)$ for all $x\in N$.

\begin{Proposition}
Let $0 \rightarrow L  \xrightarrow{\alpha}M \xrightarrow{\beta}N \rightarrow 0$ be a short exact sequence of modules over a ring $A$. Then the followings are equivalent.
\begin{itemize}
\item[(i)]
$\alpha$ is a split epimorphism.
\item[(ii)]
There is a commutative diagram
\begin{align*}
0\quad\quad&\rightarrow & L & \quad\quad\quad\xrightarrow{\alpha}& M\;\;\quad &\quad\xrightarrow{\beta} &N\quad\quad&\rightarrow& 0\\
  &		& \parallel &			& \downarrow \gamma\quad&		&\parallel\quad\quad&\quad	& \\
0\quad\quad &\rightarrow& L &\quad\quad\quad \xrightarrow{i_1} & L\oplus N \quad&\quad \xrightarrow{\pi_2}&N\quad\quad&\rightarrow&0
\end{align*}
where $i_1$ is an inclusion into the first summand and $\pi_2$ is a projection onto the second summand.
\end{itemize}
\end{Proposition}
\begin{proof}
If (ii) is satisfied, then $\beta$ is split by $\gamma^{-1}i_2$. To prove the other direction, let $\eta=(\alpha, \text{spit of }\beta)$. It suffices to show that $\eta$ is invertible. This is proved by the Five-Lemma in \cite[The Five-Lemma]{hatcher_algebraic_2001} and the Snake Lemma in \cite[Snake Lemma 1.3.2]{Weibel_1994}.
\end{proof}

For any short exact sequence $0 \rightarrow L  \xrightarrow{\alpha}M \xrightarrow{\beta}N \rightarrow 0$ with $\alpha$ being a split epimorphism, $M\cong L\oplus N$ by the isomorphism $\gamma$ in (ii).

\begin{Proposition}\label{projdef}\index{projective}
Let $A$ be a ring. The followings are equivalent:
\begin{enumerate}
\item
An $A$-module $P$ is a direct summand of a free module.
\item
Every epimorphism $V\rightarrow P$ is split.
\item For every pair of morphisms $\begin{smallmatrix} &&P \\ &&\\ && \;\downarrow\, \alpha\\ && \\ V&\rightarrow & W\\ &\beta: \text{ epi}& \end{smallmatrix}$ where $\beta$ is an epimorphism, there exists a morphism $\gamma: P \rightarrow V$ where the diagram $\begin{smallmatrix} &&P \\ &\gamma&\\ & \,\swarrow & \;\downarrow\, \alpha\\ & & \\ V&\rightarrow & W\\ &\beta: \text{ epi}& \end{smallmatrix}$ commutes.
\item For every short exact sequence of $A$-modules $0 \rightarrow V \rightarrow W \rightarrow X \rightarrow 0$, the corresponding sequence $0 \rightarrow \Hom_A(P,V) \rightarrow \Hom_A(P,W) \rightarrow \Hom_A(P,X) \rightarrow 0$ is exact.

\end{enumerate}
\end{Proposition}
\begin{proof}
The proof is omitted as it's standard.
\end{proof}

\index{restricted representation}
Let $H$ be a subgroup of $G$ and $W$ be a representation of $G$. The \textit{restricted representation}\index{restricted representation} to $H$, denoted by $(W)_H$, is a representation with its space $W$ and action restricted to $H$ where the elements of $H$ act the same way on $W$. Let $n=[G:H]$ and $x_1,\cdots, x_n$ be coset representatives of $G/H$. Then $G=\bigcup_{i=1}^n x_iH$. Let $U$ be an $kH$-module. Then the \textit{induced}\index{induced module} module $U^G$ is defined as $\bigoplus_{i=1}^n x_iU$. By comparing their dimensions,
\begin{equation}\label{inddim}
\dim U^G=[G:H]\dim U.
\end{equation}
The following lemma \cite[Lemma 8.1.2]{Webb_2016} shows induced or restricted projective modules are projective.
\begin{Lemma}\label{indrestrprojmodareproj}
Let $H$ be a subgroup of $G$.

\begin{enumerate}
\item
If $P$ is a projective $RG$-module, then $(P)_H$ is a projective $RH$-module.
\item
If $Q$ is a projective $RH$-module, then $Q^G$ is a projective $RG$-module. 
\end{enumerate}
\end{Lemma}

\begin{proof}
\begin{enumerate}
\item
From $(RG)_H \cong \oplus_{g\in H\backslash G} RHg \cong (RH)^{[G:H]}$, we derive that the restricted $RH$-module is free. $(P)_H$ is a direct summand of ${RG}^n$ on restriction to $H$, so it is a direct summand of ${RH^{[G:H]}}^n$.

\item
$(RH)\ind^G_H \cong (R\ind^H_1)\ind^G_H \cong R\ind^G_1 \cong RG$ implies that the direct summands of ${RH}^n$ are induced to the direct summands of ${RG}^n$.
\end{enumerate}
\end{proof}

\subsection{The Radical of A Regular Module}
\index{socle}
Let $A$ be a $k$-algebra. A \textit{socle}\index{socle} of $A$-module $U$, denoted by $\soc U$, is defined as the largest semisimple submodule of $U$. If $U$ is finitely generated, $\soc U$ is the sum of all simple submodules of $U$. Radical is the dual concept of socle. Radical is for when quotients of a given module. \index{radical}
The \textit{radical} of an $A$-module $U$, denoted by $\rad U$, is defined as the intersection of all maximal submodules of $U$. i.e, $\rad U= \cap\;\{\,M\; |\; M \text{ is a maximal submodule of } U\}$. Given a regular module ${}_AA$, we simply write $\rad {}_AA$ as $\rad A$. The following proposition \cite[Proposition 6.3.2, 6.3.4]{Webb_2016} gives properties of the radical and socle of a module.

\begin{Proposition}  \label{JR}\label{A/RadA semisimple}
\begin{itemize}
\item[(i)]
$\rad A$ annihilates semisimple modules. That is, $\rad A= \{a \in A \;|\; a\cdot S=0, \text{ for all $S$ simple $A$-modules}\}.$ The radical is called \textit{Jacobson radical}. 
\item[(ii)]
Let $A$ be a finite-dimensional algebra over a field.
Then $\rad A$ is the smallest left ideal of $A$ such that $A/\rad A$ is semisimple.
\item[(iii)]
For an $A$-module $U$ where $A$ is finitely generated, $\rad U=\rad A\cdot U$.
\item[(iv)]
For an $A$-module $U$ where $A$ is finitely generated, $\soc U = \{u\in U \,|\, \rad A\cdot u=0\}$.
\end{itemize}
\end{Proposition}

\begin{proof}
\begin{itemize}
\item[(i)]
Let $S$ be a simple $A$-module. For any nonzero $s\in S$, the module homomorphism $A\rightarrow S$ defined by $a\mapsto as$ is surjective. The kernel $M_s$ is a maximal left ideal, so $\rad A \subseteq M_s$. For any $a\in \rad A$, $as=0$ for all simples $S$ and $s\in S$. Conversely, let $a\in A$ such that $a\cdot S=0$ for all simple modules $S$. Fix a maximal left ideal $M$. Then $A/M$ is a simple module so $a\cdot (A/M)=0$. That is, $a\in M$ for arbitrary maximal $M$. This shows $a\in\rad A$.

\item[(ii)]
The submodule $\rad A$ of $A$ is finitely generated. This shows $\rad A= M_1 \cap\cdots\cap M_n$. The quotient $A/\rad A$ is semisimple $A$-module by .

\item[(iii)]
Let $V$ be a maximal submodule of $U$. Then $U/V$ is simple. By (i), $\rad A\cdot(U/V)=0$, so $\rad A\cdot U \subseteq V$. $\rad A\cdot U$ is independent of choice of maximal $V$, so $\rad A\cdot U \subseteq \cap\{M\,|\,M: \text{maximal submodules of }U\}=\rad U$. Conversely, $\rad(U/(\rad A\cdot U))$ is zero as $U/\rad A\cdot U$ is semisimple. Then $\rad U \subseteq \rad A\cdot U$.

\item[(iv)]
The set $\{u\in U \,|\, \rad A\cdot u=0\}$ is the largest submodule of $U$ annihilated by $\rad A$, so it is $(A/\rad A)$-module. As $A/\rad A$ is semisimple, it is semisimple and is a submodule of $\soc U$. Conversely, every semisimple submodule of $U$ is annihilated by $\rad A$ as $\rad A$ is the Jacobson radical. The largest semisimple submodule is annihilated by $\rad A$, so $\soc U$ is included in the set.
\end{itemize}\end{proof}

For each $A$-module $U$, define $\rad^n(U)=\rad(\rad^{n-1}(U))$ iteratively. By Proposition \ref{A/RadA semisimple} (iii), the submodules $\rad^n U$ of $U$ form a chain, $\cdots \subseteq \rad^2 U \subseteq \rad U \subseteq U$. The chain is a composition series (defined in the Appendix \ref{composseries}). The composition series is called the \textit{radical series}\index{radical series} of $U$. The quotients $\rad^{n-1}(U)/\rad^{n}(U)$ are called the \textit{radical layers}\index{radical layer}. The quotient $A /\rad A$ is called the \textit{head}\index{head} of $A$. Nakayama's Lemma \cite[Theorem 7.3.1]{Webb_2016} tells us that the quotient map from a module to its head is essential.

\begin{Theorem}[Nakayama's Lemma]\label{nakayamalemma}\index{Nakayama's Lemma}
Let $U$ be any Noetherian $A$-module. Then the quotient map $U \rightarrow U/\rad U$ is essential.

\end{Theorem}

\begin{proof}
Equivalently, we want to show that if $V$ is a submodule of $U$ such that $V+\rad U=U$, then $V=U$. For contradiction, assume that $V \neq U$. Then $V \subsetneq U$ so there exists a maximal submodule $M$ of $U$ such that $V \subseteq M \subsetneq U$. $\rad U \subseteq M$ by the definition of radical, so $V+\rad U \subseteq M$. Then the composite map $V \hookrightarrow U \rightarrow U/\rad U$ where $U \rightarrow U/\rad U$ is the canonical quotient map has the image that is contained in $M/\rad U$. This leads to a contradiction that $V+\rad U \subsetneq U$, as $M/\rad U \neq U/\rad U$ by $(U/\rad U)/(M/\rad U) \cong U/M \neq 0$.
\end{proof}

\subsection{Projective Cover}

Projective cover is introduced for regular module decomposition. It is an essential epimorphism with certain condition, but it is also considered as a minimal module by its uniqueness.\index{essential epimorphism} An \textit{essential epimorphism} is an epimorphism of modules $f:U \rightarrow V$ such that no proper submodule of $U$ is mapped surjectively onto $V$ by $f$. That is, $f$ is called an essential epimorphism if $g: W \rightarrow V$ is a morphism such that $fg$ is an epimorphism, then $g$ is epimorphism.

\begin{Theorem}\label{ggfessepithenfis}
Let $f: U \rightarrow V$ and $g: V \rightarrow W$ be two module homomorphisms. If $g$ and $gf$ are essential epimorphisms, then so is $f$.
\end{Theorem}
\begin{proof}
Let $g$ and $gf$ be essential epimorphisms. If $f$ is not an epimorphism, then $f(U)$ is a proper submodule of $V$. Then $gf(U)$ would be a proper submodule of $W$, as $g$ is essential. This leads to a contradiction as $gf$ is an epimorphism, so $gf(U)=W$. This shows that $f$ is an epimorphism. If $U_0$ is a proper submodule of $U$, then $gf(U_0)$ is a proper submodule of $W$ as $gf$ is essential. $f(U_0)$ is also a proper submodule of $V$ as $g$ is epimorphism, so $f$ is essential.
\end{proof}

\index{projective cover}
The morphism $f: P_U \rightarrow U$ is called the \textit{projective cover} of $U$ if $P_U$ is a projective module and $f$ is an essential epimorphism. Equivalently, the projective cover of $U$ is defined as the projective module $P_U$ if $f: P_U\rightarrow U$ is an essential epimorphism. For a finitely generated algebra, every module has a projective cover by the following theorem \cite[Theorem 7.3.10]{Webb_2016}. After the theorem, the following theorem \ref{uniqueness projective cover} \cite[Proposition 7.3.3]{Webb_2016} shows that if it exists, the projective cover is unique.

\begin{Theorem}[Existence of Projective Cover]\label{existprojcover}
Let $A$ be a finite-dimensional algebra over a field $k$ and $U$ be an $A$-module. Then $U$ has a projective cover.
\end{Theorem}

\begin{proof}
By Nakayama's Lemma, $U$ is a projective cover of $U/\rad U$. As $U/\rad U$ is semisimple, $U/\rad U=S_1\oplus\cdots\oplus S_n$ for some simples $S_1,\cdots, S_n$. Then $P_{S_1}\oplus\cdots\oplus P_{S_n} \rightarrow U/\rad U$ is the projective cover of $U/\rad U$. By projectivity, there exists a homomorphism $f$ such that the diagram
 $$\begin{matrix} && P_{S_1}\oplus\cdots\oplus P_{S_n}  \\ & &\\ & f \,\swarrow & \;\downarrow\, h\\ & & \\ U &\longrightarrow & U/\rad U \\ & g & \end{matrix}$$
 commutes. By \ref{ggfessepithenfis}, $f$ is also an essential epimorphism. This proves that $f$ is a projective cover of $U$.
\end{proof}

\begin{Theorem}[Uniqueness of Projective Cover]\label{uniqueness projective cover}
\begin{enumerate}

\item
Suppose $f: P\rightarrow U$ is a projective cover of an $A$-module $U$ and $g:Q\rightarrow U$ is an epimorphism where $Q$ is a projective module. Then $Q$ has a decomposition $Q_1\oplus Q_2$ so that $g$ has components $g=(g_1,0)$ with respect to this direct sum decomposition and $g_1: Q_1 \rightarrow U$ appears in a commutative triangle $\begin{smallmatrix} &&Q_1 \\ &\gamma&\\ & \,\swarrow & \;\downarrow\, g_1\\ & & \\ P&\rightarrow & U\\ &f: \text{ epi}& \end{smallmatrix}$ where $\gamma$ is an isomorphism.

\item
If any exist, the projective covers of a module $U$ are all isomorphic, by isomorphisms that commute with essential epimorphisms.

In other words, a projective cover is unique if it exists.

\end{enumerate}
\end{Theorem}

\begin{proof}

\begin{enumerate}

\item
P and Q are both projective, so let $\alpha: P \rightarrow Q$ and $\beta: Q \rightarrow P$ be morphisms lifted from $f$ and $g$, respectively. $\begin{smallmatrix} &&Q \\ &&\\ \alpha& \, \nearrow \swarrow \beta & \;\downarrow\, g\\ & & \\ P&\rightarrow & U\\ &f: \,\text{epi}& \end{smallmatrix}$. Then $f\beta\alpha=g\alpha=f$ is an essential epimorphism. This implies $\beta\alpha$ is an epimorphism, so $\beta$ is an epimorphism. The epimorphism $\beta$ splits as $P$ is projective. $Q$ has a decomposition $Q_1\oplus Q_2$ where $Q_2=\Ker \beta$ and $\beta$ maps $Q_1$ isomorphically to $P$. Then $g=(f\beta\vert_{Q_1}, 0)$ with an isomorphism $\beta\vert_{Q_1}$, as desired.

\item
Suppose $f: P\rightarrow U$ and $g: Q\rightarrow U$ are both projective covers. Then $Q_1$ is a submodule of $Q$ that is mapped onto $U$ and $f$ is essential, so $Q=Q_1$. The isomorphism $\gamma$ in the commutative triangle proves $P\cong Q$.
\end{enumerate}
\end{proof}

\section{Decompositions of A Regular Module}

\index{indecomposable}
A left $R$-module $U$ is said to be \textit{indecomposable}\index{indecomposable} if $U$ cannot be expressed as a direct sum of two non-trivial submodules of $U$. Explicitly, if $U \cong V \oplus W$, then $V=0$ or $W=0$.
\label{Simpleindecomp}
By the definition, simple modules are indecomposable.
For a finitely generated algebra $A$, the regular module ${}_AA$ can be decomposed as a direct sum of indecomposable modules $A_i$. i.e, ${}_AA=A_1\oplus \cdots \oplus A_n.$ The question is that what type of indecomposable submodules $A_i$ would be so that the decomposition of ${}_AA$ is unique up to isomorphism. \index{semisimple}A module $U$ is called \textit{semisimple} if $U$ can be decomposed as a direct sum of simple submodules.

\subparagraph{Primitive Idempotents} The concept of primitive idempotent is introduced for decomposing a regular module.
\index{idempotent}\index{orthogonal}\index{primitive}\index{lift}Let $A$ be a finite-dimensional $k$-algebra.
An element $e$ in $A$ is called \textit{idempotent} if $e^2=e$.
Any pair of elements $e$ and $f$ in $A$ is said to be \textit{orthogonal} if $ef=0=fe$.
We say $e$ is \textit{primitive} if it cannot be written as a sum of orthogonal idempotents. Let $A$ be a ring and $I$ be an ideal of $A$. For an idempotent $f$ of $A$, we say $f$ \textit{lifts} $e$ if $e=f+I$ is an idempotent of $A/I$. The Lemma \ref{eidempfidemp} \cite[Theorem 7.3.5]{Webb_2016} shows that any lift $e+I$ is primitive if the $e$ is primitive.

\begin{Lemma}\label{eidempfidemp}
Let $I$ be a nilpotent ideal of a ring $A$. If $e$ is an idempotent in $A/I$, then there exists an idempotent $f$ in $A$ with $e=f+I$.
Moreover, if $e$ is primitive, so is any lift $f$.
\end{Lemma}
\begin{proof}
We define idempotents $e_i\in A/I^i$ inductively such that $e_i+I^{i-1}/I^i=e_{i-1}$. Suppose $e_{i-1}$ is an idempotent of $A/I^{i-1}$. Pick any element $a\in A/I^i$ mapping onto $e_{i-1}$ so that $a^2-a\in I^{i-1}/I^i$. As $(I^{i-1})^2 \subseteq I^i$, $(a^2-a)^2$ is zero. Define $e_i=3a^2-2a^3$. Then $e_i^2-e_i=(3a^2-2a^3)(3a^2-2a^3-1)=-(3-2a)(1+2a)(a^2-a)^2=0$ so $e_i$ is an idempotent. This completes the induction definition. Pick an integer $r$ satisfying $I^r=0$. Define $f$ as $e_r$. Then the theorem is proved by induction. Assume $e$ is primitive. For contradiction, suppose $f$ can be written as $f=f_1+f_2$ where $f_1$ and $f_2$ are orthogonal idempotents. Then $e=e_1+e_2$ where $e_i=f_i+I$ and it is also a sum of orthogonal idempotents. By the assumption on $e$, either $e_1=0$ or $e_2=0$. Without loss of generality, suppose $e_1=0\in A/I$. Then ${f_1}^2=f_1\in I$. Since $I$ is nilpotent, the ideal contains no nonzero idempotent.
\end{proof}

\subparagraph{Decomposition of A Regular Module with Idempotents}
The following proposition \cite[Proposition 7.2.1]{Webb_2016} shows that in a regular module, indecomposable modules have one-to-one correspondence with primitive orthogonal idempotents.
$${}_AA \cong Ae_1\oplus\cdots\oplus Ae_n,$$
where $e_1,\cdots, e_n$ are primitive orthogonal idempotents in $A$ satisfying $1=e_1+\cdots+e_n$. In the group algebra $kG$, 
$$kG=kGe_1\oplus\cdots\oplus kGe_n.$$
where $e_1,\cdots, e_n$ are primitive orthogonal idempotent in $kG$ satisfying $1=e_1+\cdots e_n$.

For a finitely generated algebra, every idempotent can be expressed in a sum of primitive orthogonal idempotents.

\begin{Proposition}\label{regmoddecompwithidempots}
Let $A$ be a ring. Then
\begin{itemize}
\item[(i)]
The decompositions of the regular representation as a direct sum of submodules ${}_AA=A_1 \oplus \cdots \oplus A_r$ biject with expressions for the identity of $A$ as a sum of orthogonal idempotents $1=e_1+\cdots+e_r$ in such a way that $A_i=Ae_i$.

\item[(ii)]
$A_i$ is indecomposable if and only if the corresponding idempotent $e_i$ is primitive.
\end{itemize}
\end{Proposition}

\begin{proof}
\begin{itemize}
\item[(i)]
Suppose $1=e_1+ \cdots + e_r$ is an expression for the identity as a sum of orthogonal idempotents $e_i$. Then $ A= Ae_1 \oplus \cdots \oplus Ae_r$ for the submodules $Ae_i$ of $A$. For any $x\in A$, $x=x\cdot 1 = x(e_1 +\cdots + e_r)=xe_1+\cdots +xe_r$. This shows that $Ae_1 + \cdots + Ae_r=A$. If $x\in Ae_i \cap \sum_{j\neq i}Ae_j$, then $x=xe_i$ and $x=\sum_{j\neq i}a_je_j$. $x=xe_i=\sum_{j\neq i}a_je_je_i =0$, so $Ae_i \cap Ae_j=0$ for all $1 \leq i\neq j \leq r$. Coversely, suppose that $A=A_1 \oplus \cdots \oplus A_r$, a direct sum of submodules $A_i$. Then $1=e_1 +\cdots + e_r$ where $e_i \in A_i$ is a uniquely determined element. $e_i = e_i 1= e_ie_1+\cdots+e_ie_r$ is an expression in which $e_ie_j\in A_j$. This is because $A_j$ is a submodule, so $e_ie_j\in A_j$. The only such expression of $e_i$ is $e_i$, because $e_i\notin A_j$ for $i\neq j$ and $e_i \in A_i$. Thus, $e_ie_j=0$ for $i\neq j$ and $e_ie_i=e_i\in A_i$. This shows that the components $e_1, \cdots, e_r$ are orthogonal idempotents.
\item[(ii)]
(i) shows that an expression for 1 as a sum of idempotents corresponds to a module direct sum decomposition into submodules. This means that a summand $A_i$ decomposes as a direct sum of two other summands if and only if an orthogonal idempotent $e_i$ decomposes as a sum of two other orthogonal idempotents. This is equivalent to the statement that $A_i$ is indecomposable if and only if $e_i$ is primitive.
\end{itemize}
\end{proof}

\subsection{A Decomposition into Projective Indecomposable Modules}

In a regular module ${}_AA$ where $A$ is finite-dimensional algebra, projective indecomposable modules have one-to-one correspondence with simple modules. This can be proved by \ref{regmoddecompwithidempots} that projective indecomposable modules biject with primitive idempotents in $A$. The following theorem and proof are based on \cite[Theorem 7.3.8]{Webb_2016}.

\begin{Theorem}\label{Onetoonecorrespondenceprojectivesimple}

Let $A$ be a finite-dimensional algebra over a field and $S$ be a simple $A$-module. Then

\begin{itemize}
\item[(i)]
There exists an indecomposable projective module $P$ such that $P/\Rad P \cong S$, and $P$ is of the form $Af$ where $f$ is a primitive idempotent in $A$.

The idempotent $f$ has the property that $fS \neq 0$ and $fT=0$ if $T \ncong S$.

\item[(ii)]
$P$ is the projective cover of $S$, i.e., $P=P_S$.

 $P$ is uniquely determined up to isomorphism.
 
 $P$ has $S$ as its unique simple quotient.

\end{itemize}

\end{Theorem}

\begin{proof} 
\begin{itemize}
\item[(i)]
By \ref{A/RadA semisimple} and \ref{regmoddecompwithidempots}, $1=e_1+\cdots+e_n$ where $e_i$ are primitive orthogonal idempotents in $A/\rad A$. Then $S= 1\cdot S=(e_1+\cdots+e_n)\cdot S= e_1\cdot S \oplus \cdots \oplus e_nS$. $S$ is simple, so there exists $e_k$ such that $e_kS=S$ and $e_iS=0$ for $i\neq k$. Define $e$ as $e_k$. This shows that the primitive idempotent $e$ in $A/\Rad A$ acts nonzero on a unique isomorphism class of simple modules. Any lift $f$ of $e$ to $A$ is primitive by \ref{eidempfidemp}, $fS=S$ and $fT=0$ if $T \ncong S$. The map $\Psi: Af \rightarrow (A/\rad A)\cdot(f+\rad A)$ defined by $\Psi(af)= af+\rad A$ has the kernel $(\rad A)\cdot f$ and $Af/(\rad A\cdot Af) \cong (A/\rad A)\cdot(f+ \rad A)$ by the first isomorphism theorem. $\rad Af=\rad A\cdot Af$ and $(A/\rad A)\cdot(f+ \rad A)=(A/\rad A)\cdot e=S$. $Af/\rad Af\cong S$. Define $P$ as $Af$. By one-to-one correspondence in \ref{regmoddecompwithidempots}, $P$ is indecomposable. $P$ satisfies $P/\rad P \cong S$.

\item[(ii)] By Nakayama's Lemma \ref{nakayamalemma}, $P$ is the projective cover of $S$, $P_S$. By the uniqueness of projective cover in \ref{uniqueness projective cover}, $P$ is uniquely determined up to isomorphism. Any simple quotient of $P$ is a quotient of $P/\rad P$ that is isomorphic to $S$. The simple quotient of $P$ is $S$.
\end{itemize}
\end{proof}

Schur's Lemma tells that for two non-isomorphic simples, no homomorphism exists from one to another. Because the lemma is related to homomorphism spaces, it is sometimes used when we want to find submodules or quotients of a module or the multiplicity of a simple in a module.

\begin{Proposition}[Schur's Lemma]
Let $R$ be a ring and $S$ and $T$ be simple left (or right) $R$-modules. Then $\End_R(S)$ is a division ring and if $\Hom_R(S,T)\neq 0$, then $S\cong T$.

Over an algebraically closed field $k$, $\End_R(S)\cong k$.
\end{Proposition}
\begin{proof}
Let $\theta$ be a nonzero homomorphism in $\End_R(S)$. Then $\im(\theta)$ is nonzero and is a submodule of the simple $S$, so it is $S$. This implies $\theta$ is surjective. As $\theta$ is not a zero map, $\ker(\theta)$ is a proper submodule of $S$, so is zero. This implies $\theta$ is injective. It follows that $\theta$ is a bijection. Suppose $\Hom_R(S,T)$ is nonzero. Then there exists a nonzero homomorphism $f\in \Hom_R(S,T)$. $\im(f)$ is a submodule of $T$, $0\neq \im(f)\leq T$. As $T$ is simple, $\im(f)=T$. $f$ is nonzero implies $\ker(f)$ is a proper submodule of $S$, so $\ker(f)=0$. This shows $f$ is an isomorphism, so $S\cong T$.
\end{proof}

\begin{Example}
$\mathbb{Z}_4$ is not a simple $\mathbb{Z}$-module by the Schur's Lemma, as the morphism $\theta: \mathbb{Z} \rightarrow \mathbb{Z}_4$ defined by $\theta(x)=2x$ is not invertible, so $\End_{\mathbb{Z}}(\mathbb{Z}_4)$ is not a division ring.
\end{Example}

\begin{Theorem}[Artin-Wedderburn Theorem]\label{Artinwedderburnthm}\index{Artin-Wedderburn Theorem}

Let $A$ be a finite-dimensional algebra over a field $k$ so that every finite-dimensional module is semisimple.
Then $A$ is a direct sum of matrix algebras over division rings. That is,
$$A \cong M_{n_1}(D_1)\oplus \cdots \oplus M_{n_r}(D_r)$$
for some $n_1, \cdots, n_r \in \mathbb{N}$ and division rings $D_1, \cdots, D_r$, which are uniquely determined up to ordering.

The matrix algebra summands are uniquely determined as subsets of $A$.

\end{Theorem}

\begin{proof}

We use the fact that every such direct sum of matrix algebras is semisimple. In other words, each matrix algebra over a division ring has a unique simple module up to isomorphism. If there exists a direct sum decomposition $U=U_1 \oplus \cdots \oplus U_r$ of a module $U$, then $\End_A(U)$ is isomorphic to the algebra of $r\times r$ matrices $M_{ij}$ in which $M_{ij} \in \Hom_A(U_i, U_j)$. Any endomorphism $\phi: U \rightarrow U$ may be written as a matrix of components $\phi=(\phi_{ij})$ where $\phi_{ij}: U_j \rightarrow U_i$. In this perspective, composition of endomorphisms is viewed as matrix multiplication. This correspondence will be used in decomposition of $A$.

By Schur's Lemma, $\Hom_A(S_j^{n_j}, S_i^{n_i})=0$ for $i \neq j$. This shows that $\End_A(A) \cong \End_A(S_1^{n_1}) \oplus \cdots \oplus \End_A(S_r^{n_r})$ and $\End_A(S_i^{n_i}) \cong M_{n_i}(D_i^{\op})$. Evidently, $M_{n_i}(D_i^{\op})^{\op} \cong M_{n_i}(D_i)$. By  , this implies $\End_A(A) = A^{\op}$. This shows the theorem, as desired.

If $k$ is algebraically closed, by Schur's Lemma, $D_i=k$ for all $i$.
\end{proof}

We've seen that given a simple module, there is a projective indecomposable module that is exactly the projective cover of the given simple module. In fact, there is a one-to-one correspondence between projective indecomposable modules and simple modules, that is, each projective indecomposable module is exactly the projective cover of a simple module and is uniquely determined up to isomorphism. The one-to-one correspondence is showed by the following theorem. Our next theorem is important throughout the topic, in that the theorem gives a regular module decomposition ${}_AA=A_1 \oplus \cdots \oplus A_n$ for projective indecomposable modules $A_i$ and shows the decomposition is uniquely determined up to isomorphism. In the proof, Artin-Wedderburn Theorem \cite[Theorem 2.1.3]{Webb_2016} is used on $A/\rad A$ as it is semisimple. We use this theorem for decomposing given group algebras into a direct sum of projective indecomposable modules in chapter 4, 5 and 6. The following theorem and proof are based on \cite[Theorem 7.3.9]{Webb_2016}.

\begin{Theorem}\label{regmodAdecompassumofprojcovers}

Let $A$ be a finite-dimensional algebra over a field $k$.
\begin{enumerate}
\item Each projective indecomposable $A$-module is the projective cover of a simple module and is uniquely determined up to isomorphism of simple.

That is, if $Q$ is a projective indecomposable module, then $Q=P_S$ for some simple module $S$. Moreover, $P_S \cong P_T$ if and only if $S \cong T$.

\item Each projective cover of simple module is a direct summand of ${}_AA$ where its multiplicity is the multiplicity of $S$ as a summand of $A/\Rad A$. In other words, ${}_AA$ is decomposed as 
$${}_AA \cong \bigoplus_{\text{simple } S} {P_S}^{n_S}$$
where $n_S=\dim_k S$ if $k$ is algebraically closed. In general, $n_S=\dim_D(S)$ where $D=\End_A(S)$.
\end{enumerate}

\end{Theorem}

\begin{proof}
\begin{enumerate}
\item
Let $P$ be a projective indecomposable module. By \ref{A/RadA semisimple}, $P/\Rad P$ is semisimple, so the radical quotient can be written as a direct sum of simples, i.e, $P/\Rad P \cong S_1 \oplus \cdots \oplus S_n$ for some simples $S_1,\cdots, S_n$. By Nakayama's Lemma \ref{nakayamalemma}, the morphism $P \rightarrow S_1 \oplus \cdots \oplus S_n$ is an essential epimorphism. The morphism $P_{S_1} \oplus  \cdots \oplus P_{S_n} \rightarrow S_1 \oplus \cdots \oplus S_n$ is also an essential epimorphism. This implies that $P$ and $P_{S_1} \oplus  \cdots \oplus P_{S_n}$ are both the projective covers of $S_1 \oplus \cdots \oplus S_n$, so $P \cong P_{S_1}\oplus \cdots \oplus P_{S_n}$ by the uniqueness of projective covers. As $P$ is indecomposable, $n=1$. Without loss of generality, $P=P_S$ for some simple module $S$.
\item
We consider the radical of $A$, $A/\Rad A$. As the radical is semisimple by \ref{A/RadA semisimple}, we use Artin-Wedderburn Theorem for its decomposition. By the theorem, $$A/\Rad A \cong \bigoplus_{\text{simple } S} S^{n_S}$$
where $n_S$ is $\dim_D(S)$ such that $D=\End_A(S)$.
The above isomorphic relation tells that $n_S$ is also the multiplicity of simple $S$ in $A/\Rad A$.

By Nakayama's Lemma and by the finite-dimensional algebra $A$ being projective, $A$ is the projective cover of $A/\rad A$. Both $A$ and $\bigoplus_{\text{simple }S}{P_S}^{n_S}$ are the projective covers of $A/\Rad A$, so they are isomorphic by the uniqueness of projective covers. Therefore, we obtain the decomposition $${}_AA \cong \bigoplus_{\text{simple } S} {P_S}^{n_S}$$
where $n_S= \dim_D(S)$ and $D=\End_A(S)$. $n_S=\dim_kS$ when $k$ is algebraically closed, as $D=k$ by Schur's Lemma.
\end{enumerate}
\end{proof}

\label{dimAbydecompproj}
From the above theorem \ref{regmodAdecompassumofprojcovers}, $\dim A=\sum_{\text{simple}\,S} n_S\cdot \dim P_S$. If $k$ is algebraically closed, then $\dim A=\sum_{\text{simple}\,S} \dim_k S\cdot \dim P_S.$ If we write projective indecomposable modules $P_i$ for $i=1,\cdots, n$ in $A$ such that $P_i=P_{S_i}$, then $\dim A=\sum_{i=1}^n \dim_k S_i \cdot \dim P_i$
where $n$ is the number of simple modules in $A$. We will use this in chapter 4, and 5 for when calculating the dimensions of projective indecomposable modules in a given group algebra.

\section{Decomposition into Blocks}\chaptermark{Blocks of $kG$}

In the previous section, a regular module can be decomposed as a direct sum of projective indecomposable modules. A regular module not only has such decomposition but also, as a ring, can be decomposed as a direct sum of indecomposable two-sided ideals which are generated by primitive central orthogonal idempotents. Moreover, the decomposition is unique as we'll show below.
Let $k$ be an algebraically closed field. The \textit{center}\index{center} of $Z(A)$ is the set $\{z\in A \,|\, za=az \text{   for all } a\in A \}=\{z\in A \,|\, a^{-1}za=a \text{   for all } a\in A \}$. A \textit{central}\index{central} element $e$ in a ring $R$ is an element which commutes with all elements in $R$, so $e\in Z(R)$. Explicitly, $er=re$ for all $r\in R$.

\paragraph{A (Ring) Decomposition into Blocks}
The following theorem \cite[Proposition 3.6.1]{Webb_2016} tells that decompositions as a ring into a direct sum of two-sided ideals biject with expressions of the identity as a sum of orthogonal central idempotents.

\begin{Theorem}\label{ringdecompcentralorthogidemp}
Let $A$ be a ring with identity $1$. Then
\begin{itemize}
\item[(i)]
Decompositions $A=A_1\oplus\cdots\oplus A_r$ as a direct sum of two-sided ideals $A_i$ biject with expressions $1=e_1+\cdots+e_r$ as a sum of orthogonal central idempotents in such a way that $e_i$ is the identity element of $A_i$ and $A_i=Ae_i$.

\item[(ii)]
$A_i$ are indecomposable as rings if and only if the corresponding $e_i$ are primitive central idempotent elements.

\item[(iii)]
If every $A_i$ is indecomposable as a ring, then the subsets $A_i$ and the corresponding primitive central idempotents $e_i$ are uniquely determined.

This implies that primitive central orthogonal idempotents are unique.

\item[(iv)]
Every central idempotent can be written as a sum of the primitive central orthogonal idempotents $e_i$.

\end{itemize}
\end{Theorem}

\begin{proof}
\begin{itemize}
\item[(i)]
Given any ring decomposition $A=A_1\oplus \cdots \oplus A_r$ as a direct sum of two-sided ideals $A_i$, $1=e_1+\cdots+e_r$ for each $e_i\in A_i$. From the equation, $e_i= \sum_{j=1}^r e_ie_j$. $e_ie_j$ are in the two-sided ideals $A_j$ and $A_j$, so $e_ie_j$ is zero for $i\neq j$ and $e_i={e_i}^2$. Conversely, given an expression $1=e_1+\cdots+e_r$ where $e_i$ are orthogonal central idempotent elements, $A= Ae_1+\cdots+Ae_r$. If $a=a_ie_i=a_je_j\in Ae_i\cap Ae_j$ for $i\neq j$, then $ae_j=a_ie_ie_j=0$, so $a=0$. This shows $Ae_i\cap Ae_j=0$ for $i\neq j$.

\item[(ii)]
The ring $A_i$ is indecomposable means that it cannot be expressed as a direct sum of two-sided ideals. By (i), the corresponding idempotent $e_i$ cannot be expressed as a sum of orthogonal central idempotents. The other direction is also true by (i).

\item[(iii)]
It suffices to show that there is at most one decomposition of $A$ as a sum of indecomposable two-sided ideals. Suppose that there are two such decompositions and denote corresponding primitive central idempotent elements as $e_i$ and $f_j$, respectively. Then $1=e_1+\cdots+e_r=f_1+\cdots+f_s$ and $e_i=e_i\cdot 1=\sum_{j=1}^s e_if_j$. $e_if_j$ are orthogonal idempotents by the following. As $e_i$ and $f_j$ are central, $(e_if_j)(e_if_k)={e_i}^2f_jf_k=e_if_jf_k$. $(e_if_j)(e_if_k)$ is zero for $k\neq j$ and $e_if_j$ if $k=j$. By primitivity of $e_i$, $e_i=e_if_j$ for some unique $j$ and $e_if_k=0$ for $k\neq j$. Similarly, $f_j=1\cdot f_j=\sum_{l=1}^r e_lf_j$, so by primitivity of $f_j$, $f_j=e_lf_j\neq 0$ for some unique $l$. As $e_if_j\neq 0$, $l=i$. It follows that $e_i=f_j$, so $\sum_{k\neq i}e_k=\sum_{k\neq j}f_k$. $A\sum_{k\neq i}e_k=A\sum_{k\neq j}f_k$. We use induction on $r$. Clearly, it's true when $r=1$. Assume $r>1$ and (iii) holds for small $r$. Then, by induction, $e_k=f_k$.
\item[(iv)]
If $e$ is any central idempotent and $e_i$ are primitive central idempotents satisfying $1=\sum_{i=1}^r e_i$, then $ee_i$ is either $e_i$ or $0$ as $e=ee_i+e(1-e_i)$ is a sum of orthogonal central idempotents. $e=e\sum_{k=1}^re_i=\sum_{k=1}^ree_i$ is a sum of certain of the $e_i$.
\end{itemize}
\end{proof}

The above theorem implies that if $1=e_1+\cdots+e_n$ where $e_i$ are primitive central orthogonal idempotents, then the corresponding decomposition is $A=Ae_r\oplus\cdots\oplus Ae_n$ where $Ae_i$ are indecomposable ideals and the decomposition is uniquely determined. Here, $Ae_i$ are called \textit{blocks} of $A$. As a ring $A$, the decomposition as the direct sum of blocks is uniquely determined. A \textit{block}\index{block} of a ring $A$ is a direct summand $Ae_i$ of a ring $A$ where $e_i$ is a primitive central orthogonal idempotent such that $1=e_1+\cdots+e_n$. A \textit{block idempotent} is $e_i$ if $1=e_1+\cdots+e_n$ where $e_i$ are primitive central orthogonal idempotents.

We define $\modd A$\index{$\modd$} as the set of finite dimensional left $A$-modules and ${}_{A}M$ as a left $A$-module $M$. For any ${}_{A}M$, $M=1\cdot M=(e_1+\cdots+e_n)\cdot M=e_1\cdot M\oplus\cdots\oplus e_n\cdot M$. $e_i\cdot M$ is a left $Ae_i$-module. This implies that a left $A$-module can be decomposed as the direct sum of left $Ae_i$-modules ${}_{Ae_i}M$. If $M$ is indecomposable, then $M=e_i\cdot M$ for some $i$ and $M\neq e_j\cdot M$ for $j\neq i$. We say M \textit{belongs to}\index{belongs to a block} $Ae_i$ if $M=e_i\cdot M$. $M$ is indeed a left $Ae_i$-module. Since the trivial module $T$ (definition \ref{trivial representation}) is indecomposable, $T$ belongs to a certain block, say $Ae_t$. This implies $T$ is a left $Ae_t$-module. The \textit{principal block}\index{principal block} $B_0(A)$ is the block $Ae_t$ to which the trivial module $T$ belongs.

\subparagraph{Blocks of $kG$}\label{blocksandmultiplicity}
 
By \ref{regmoddecompwithidempots} and \ref{ringdecompcentralorthogidemp}, a regular module can be decomposed as the direct sum of blocks generated by primitive orthogonal central idempotents and the decomposition is uniquely determined. We want to decompose the group algebra $kG$ (over characteristic $p$) into blocks as a ring. We can write it abstractly. Let $1=e_1+\cdots+e_n$ where $e_i$ are primitive central orthogonal idempotents. Then $kG=kGe_1\oplus\cdots\oplus kGe_n$ and the decomposition is uniquely determined. The trivial module is $k_G$ is a left $kGe_t$-module for some $e_t$. $B_0(kG)$ is the block $kGe_t$.

For each primitive central idempotent $e_i$, $e_i=f_{i,1}+\cdots+f_{i,t_i}$ where $f_{i,j}$ are primitive orthogonal idempotents. $1=(f_{1,1}+\cdots+f_{1,t_1})+\cdots+(f_{n,1}+\cdots+f_{n,t_n})$ and $e_if_{i,j}=f_{i,j}$. The block $kGe_i$ can be decomposed as $kGe_i=kGf_{i,1}\oplus\cdots\oplus kGf_{i,t_i}$ where $kGf_{i,j}$ are left $kG$-modules. Each $kGf_{i,j}$ is a projective indecomposable module by \ref{regmodAdecompassumofprojcovers}.
\begin{equation*}
kG=kGf_{1,1}\oplus \cdots\oplus kGf_{1,t_1}\oplus\cdots\oplus kGf_{n,1}\oplus\cdots\oplus kGf_{n,t_n}.
\end{equation*}

The above decomposition is the decomposition into projective indecomposable modules, as in \ref{regmodAdecompassumofprojcovers}. The next theorem gives a way to find the number of blocks of $kG$.

\begin{Theorem}
Let $P$ and $Q$ be projective indecomposable $kG$-modules. Then $P$ and $Q$ are in the same block if and only if $\Hom_{kG}(P, Q)$ is nonzero.
\end{Theorem}

\begin{proof}
Let $e_1,\cdots, e_n$ be block idempotents. 
We first want to show that $kGe_i$ and $kGe_j$ are different blocks if and only if $\Hom_{kG}(kGe_i,kGe_j)=0$.
The regular module $kG$ can be decomposed as the direct sum of blocks, $kG=kGe_1\oplus\cdots\oplus kGe_n$.
\begin{align*} 
  \bigoplus_{i=1}^n{kG}e_i^{\op}&={kG}^{\op}\\ & \cong  \End_{kG}({kG}) =\Hom_{kG}({kG}e_1\oplus\cdots\oplus {kG}e_n, {kG}e_1\oplus\cdots\oplus {kG}e_n).\\
&=\bigoplus_{i=1}^n\End_{kG}({kG}e_i)\oplus\{\bigoplus_{1\leq i\neq j \leq n}\Hom_{kG}({kG}e_i,{kG}e_j)\}\\ 
&\cong \bigoplus_{i=1}^n{kG}e_i^{\op}\oplus\{\bigoplus_{1\leq i\neq j\leq n}\Hom_{kG}({kG}e_i,{kG}e_j)\}.
\end{align*}
It follows that $\oplus_{1\leq i\neq j\leq n}\Hom_{kG}(kGe_i,kGe_j)=0$, so each direct summand is zero.
\begin{equation}\label{diffblockshomzero}
\Hom_{kG}(kGe_i,kGe_j)=0 \quad\quad \text{for}\;\; i\neq j.
\end{equation}
\ref{diffblockshomzero} shows that if $kGe_i$ and $kGe_j$ are different blocks, then $\Hom_{kG}(kGe_i, kGe_j)=0$.

To prove the other direction, suppose $\Hom_{kG}(kGe_i, kGe_j)=0$. For contradiction, assume that $kGe_i$ and $kGe_j$ are the same block. Then $e_i=e_j$. The assumption is rewritten as $\End_{kG}(kGe_i)=0$. Note that $kGe_i=kGf_1\oplus\cdots\oplus kGf_n$ where $f_1,\cdots, f_n$ are primitive orthogonal idempotents. Then
\begin{equation*}
\End_{kG}(kGe_i)=\Hom_{kG}(kGf_1\oplus\cdots\oplus kGf_n, kGf_1\oplus\cdots\oplus kGf_n)
\end{equation*}
 Each $kGf_i$ are projective indecomposable $kG$-modules, say $P_i$. Let $S_i$ be the simple module such that $P_i=P_{S_i}$. The multiplicity of $S_i$ in $P_j$ is $\dim_k\Hom_{kG}(P_i, P_j)=\dim_k\Hom_{kG}(kGf_i, kGf_j)$. Since $P_i/\rad P_i\cong \soc P_i\cong S_i$, the multiplicity of $S_i$ in $P_i$ is nonzero, so $\dim_k\End_{kG}(P_i)\neq0$. $\End_{kG}(P_i)=\End_{kG}(kGf_i)$ is nonzero, so $\End_{kG}(kGe_i)$ is nonzero, which leads to a contradiction.

Let $P$ and $Q$ be projective indecomposable modules and $S$ and $S'$ be simple modules satisfying $P=P_{S}$ and $Q=P_{S'}$. Suppose $P$ and $Q$ are in different blocks $kGe$ and $kGe'$, respectively. By \ref{diffblockshomzero}, $0=\Hom_{kG}(kGe,kGe')=\Hom_{kG}(P\oplus V,Q\oplus W)=\Hom_{kG}(P,Q)\oplus\Hom_{kG}(P,W)\oplus\Hom_{kG}(V,Q)\oplus\Hom_{kG}(V,W)$, so $\Hom_{kG}(P,Q)=0$.

Conversely, we want to show if $\Hom_{kG}(P,Q)=0$, then $P$ and $Q$ are in different blocks $kGe_i$ and $kGe_j$, respectively. That is, we want to show that $\Hom_{kG}(kGe_i, kGe_j)=0$. For contradiction, assume $\Hom_{kG}(kGe_i, kGe_j)\neq0$. Then there exists a projective indecomposable module $R$ such that $R$ is in both blocks $kGe_i$ and $kGe_j$, which leads to a contradiction.
\end{proof}

Since the multiplicity of $S$ in $Q$ is $\dim_k(\Hom_{kG}(P, Q))$, if the multiplicity of $S$ in $Q$ is nonzero, then $P$ and $Q$ are in the same block.

\chapter{The Structure of Projective Indecomposable $kG$-module}\chaptermark{Radical Series Structure}

\section{An Injective Envelope}

An injective module is the dual concept of projective module. \index{split monomorphism}Let $A$ be a ring. An $A$-module homomorphism $f:M \rightarrow N$ is called \textit{split monomorphism} if there exists an $A$-module homomorphism $g:N \rightarrow M$ such that $gf=1_M$. If $f:M \rightarrow N$ is a split monomorphism, then it is a monomorphism. For any $f(x)=f(y)$, $x=g(f(x))=g(f(y))=y$. The following lemma and theorem refer to \cite[Proposition 8.5.2, Corollary 8.5.3.(1)]{Webb_2016}.

\begin{Proposition}\index{injective}
Let $A$ be a ring. The followings are equivalent:
\begin{enumerate}
\item
An $A$-module $I$ is injective.
\item
Every monomorphism $I \rightarrow V$ is split.
\item For every pair of morphisms $\begin{smallmatrix} &&I \\ &&\\ && \;\uparrow\, \alpha\\ && \\ V&\leftarrow & W\\ &\beta: \text{mono}& \end{smallmatrix}$ where $\beta$ is an monomorphism, there exists a morphism $\gamma: V \rightarrow I$ where $\begin{smallmatrix} &&I \\ &\gamma&\\ & \,\nearrow & \;\uparrow\, \alpha\\ & & \\ V&\leftarrow & W\\ &\beta: \text{ mono}& \end{smallmatrix}$.
\end{enumerate}
\end{Proposition}

\begin{proof}
It follows from the duality of projective and injective module and by \ref{projdef}.
\end{proof}

\begin{Lemma}\label{dualityproj}
If $P$ is projective, then $P^*$ is projective.
\end{Lemma}
\begin{proof}
Note that ${P^*}^*\cong P$ as $kG$-modules. 
If $P$ is a summand of $kG^n$, then $P^*$ is a summand of $(kG^n)^*\cong kG^n$, so $P$ is also projective.
\end{proof}

\begin{Theorem}\label{projareinj}
Finitely generated projective $kG$-modules are finitely generated injective $kG$-modules.
\end{Theorem}
\begin{proof}
Let $P$ be a finitely generated projective $kG$-module. Suppose that there exists a morphism $\begin{smallmatrix} &&P \\ &&\\ && \;\uparrow\, \alpha\\ && \\ V&\leftarrow & W\\ &\beta: \text{mono}& \end{smallmatrix}$ with $\beta$ being injective. Then the dual is $\begin{smallmatrix} &&P^* \\ &&\\ && \;\downarrow\, \alpha^*\\ && \\ V^*&\rightarrow & W^*\\ &\beta^*: \text{ epi}& \end{smallmatrix}$ with surjective $\beta^*$. By projectivity of $P^*$ from \ref{dualityproj}, there exists $f: P^*\rightarrow V^*$ such that $\beta^*f=\alpha^*$.$$\begin{smallmatrix} &&P^* \\ &{}^{\exists}f&\\ & \,\swarrow & \;\downarrow\, \alpha^*\\ & & \\ V^*&\rightarrow & W^*\\ &\beta^*: \text{ epi}&& \end{smallmatrix}$$
where $\beta^*$ is surjective. By duality,  $$\begin{smallmatrix} && P \\ & f^*&\\ & \,\nearrow & \;\uparrow\, \alpha\\ & & \\ V&\leftarrow & W\\ &\beta: \text{ mono}& \end{smallmatrix},$$
and this implies that $P$ is an injective module. The other direction can be proved in a similar way.
\end{proof}

Injective envelope is the orthogonally dual concept of projective cover. \index{essential monomorphism}An \textit{essential monomorphism} is a monomorphism of modules $f: U \rightarrow V$ such that whenever $g: U\rightarrow W$ is a map satisfying $gf$ is a monomorphism then $g$ is a monomorphism. \index{injective envelope}An \textit{injective envelope}\index{injective envelope} of a module $U$ is an essential monomorphism $U \rightarrow I_U$ where $I_U$ is an injective module. For a finitely generated algebra, an injective envelope $U \rightarrow I_U$ of $U$ always exists and is unique by the duality of the theorems \ref{existprojcover} and \ref{Onetoonecorrespondenceprojectivesimple}.

\section{The Multiplicity of A Simple in A Module}

We have a concept relevant to the structure of radical series for a module, the multiplicity of a simple module $S$ as a composition factor of an arbitrary module $U$ with a composition series. The multiplicity of a simple module in a module is related to homomorphism spaces of modules. Homomorphism spaces of modules are useful when finding submodules or quotients of a module.
If $0 = U_0 \subseteq U_1 \cdots \subseteq U_n=U$ is any composition series of $U$, then the number of quotients $U_i/U_{i-1}$ that are isomorphic to $S$ is determined independently of the choice of composition series by Jorden-H\"older Theorem \ref{JordHoldThm}. \index{(composition factor) multiplicity of $S$} The \textit{(composition factor) multiplicity} of $S$ in $U$ is the number of quotients $U_i/U_{i-1}$ that are isomorphic to $S$. The following proposition refers to \cite[Proposition 7.4.1]{Webb_2016}.

\begin{Proposition}\label{multiplicitySofUisdimhomPStoU}
Let $S$ be a simple module for a finite-dimensional algebra $A$ with projective cover $P_S$ and $U$ be a finite-dimensional $A$-module. Then
\begin{itemize}
\item[(i)]
If $T$ is a simple $A$-module, then
\begin{displaymath}
   \dim_k\Hom_A(P_S,T) = \left\{
    \begin{array}{lr}
      \dim_k\End_A(S)& \;\text{if } S \cong T\\
      0 & \;\text{otherwise}.
    \end{array}
  \right.
\end{displaymath}

\item[(ii)]

The multiplicity of $S$ as a composition factor of $U$ is $$\dim_k\Hom_A(P_S,U)/\dim_k\End_A(S).$$

If $k$ is algebraically closed, then $\dim_k\End_A(S)=1$ by Schur's Lemma, so the multiplicity of $S$ in $U$ is $\dim_k\Hom_A(P_S, U)$.

\end{itemize}
\end{Proposition}

\begin{proof}
\begin{itemize}
\item[(i)]
If $P_S \rightarrow T$ is a nonzero homomorphism, then $\Rad P_S$ is in the kernel, as it is a submodule of $P_S$. As $P_S/\Rad P_S \cong S$ is simple, the kernel is $\Rad P_S$, and $S \cong T$ follows. Every homomorphism $P_S \rightarrow S$ is the composite map $P_S \rightarrow P_S/\Rad P_S \rightarrow S$. The map $P_S \rightarrow S$ is either isomorphism of $P_S/\Rad P_S$ with $S$ or zero. This shows $\Hom_A(P_S,S) \cong \End_A(S)$.

\item[(ii)]
Let $0=U_0 \subseteq U_1 \subseteq \cdots \subseteq U_n=U$ be a composition series of $U$. We use induction on the composition length $n$. When $n=1$, $U$ is simple and the multiplicity of $S$ is $1$ if and only if $S$ is isomorphic to $U$. By (i), $\dim_k\Hom_A(P_S, U)/\dim_k\End_A(S)=1$ if and only if $S\cong U$. Suppose $n>1$ and assume that the statement is true where the composition length is less than $n$. By \ref{projdef}, the exact sequence $0 \rightarrow U_{n-1} \rightarrow U \rightarrow U/U_{n-1} \rightarrow 0$ gives rise to the exact sequence $0 \rightarrow \Hom_A(P_S, U_{n-1})\rightarrow \Hom_A(P_S, U)\rightarrow \Hom_A(P_S, U/U_{n-1})\rightarrow 0$. Then $\dim_k\Hom_A(P_S,U) =\dim_k\Hom_A(P_S,U_{n-1})+\dim_k\Hom_A(P_S, U/U_{n-1})$. The result follows by dividing the both sides of the equation by $\dim\End_A(S)$ as $U_{n-1}$ and $U/U_{n-1}$ both have their composition lengths less than $n$.
\end{itemize}
\end{proof}

\index{Cartan matrix}
Let $A$ be a finite-dimensional algebra. For each pair of simple $A$-modules $(S,T)$, the \textit{Cartan invariant}\index{Cartan invariant} of $A$ is the non-negative integer $C_{ST}$ defined as the composition factor multiplicity of $S$ in $P_T$.
The \textit{Cartan}\index{Cartan matrix} matrix of $A$ is a matrix $C=(C_{ST})$ with rows and columns indexed by isomorphism types of simple $A$-modules. When finding the structure of a finite-dimensional algebra, Cartan matrix is useful as it shows the multiplicity of simple modules in its composition series. By \ref{multiplicitySofUisdimhomPStoU}, $$C_{ST}=\dim_k\Hom_A(P_S,P_T)/\dim_k\End_A(S).$$ If $k$ is algebraically closed, then $C_{ST}=\dim_k\Hom_A(P_S,P_T)$. The following proposition refers to \cite[Theorem 8.5.7, Lemma 3.2.1]{Webb_2016}.

\begin{Proposition}\label{dimhompsptisdimhomptps}
For each pair of simple $kG$-modules $S$ and $T$, 
$$\dim_k\Hom_{kG}(P_S, P_T)=\dim_k\Hom_{kG}(P_T, P_S).$$
\end{Proposition}
\begin{proof}
 We want to prove that $\Hom_R(V,W)^G=\Hom_{RG}(V,W)$. That is, an $R$-linear map $f:V\rightarrow W$ is precisely a morphism of $RG$-modules if it commutes with the action of $G$. $f:V\rightarrow W$ is fixed by $G$-action if and only if $gf(g^{-1}v)=f(v)$ for all $g\in G$ and $v\in V$. This is equivalent to $f(g^{-1}v)=g^{-1}f(v)$ for all $g\in G$ and $v\in V$. Since $g^{-1}$ runs over $G$ as $g$ does, the equation is exactly $f(gv)=gf(v)$ for all $g\in G$ and $v\in V$. This is precisely the definition of $f\in \Hom_{kG}(V,W)$, so we're done. By this relation, $\Hom_{kG}(P_S,P_T)=\{\Hom_k(P_S,P_T)\}^G\cong({P_S}^*\otimes_k P_T)^G$.
\end{proof}

\begin{Remark}\label{cartanmatrixentryandpims}
\begin{enumerate}
\item
For a group algebra $kG$, its Cartan matrix is symmetric.
\item The $(i,j)$th entry $c_{ij}$ of the Cartan matrix is the multiplicity of $S_i$ on $P_{S_j}$. It is $\dim_k\Hom_{kG}(P_{S_i}, P_{S_j})=\dim_k\Hom_{kG}(P_{S_j}, P_{S_i})$ by \ref{multiplicitySofUisdimhomPStoU} and \ref{dimhompsptisdimhomptps}. Let $P_i=P_{S_i}$ and $P_j=P_{S_j}$.
$$c_{ij}=\dim_k(\Hom_{kG}(P_i, P_j))=\dim_k(\Hom_{kG}(P_j, P_i)).$$
\end{enumerate}
\end{Remark}

\section{The Radial Series Structure of A Projective Indecomposable $kG$-Module}\chaptermark{Projective Indecomposable $kG$-Modules}

In the radical series $0 \subseteq \rad^{n-1} U \subseteq \cdots \subseteq \rad U\subseteq U$ of $U$, $\soc U$ is $\rad^{n-1}U$ as $\soc U = \{u\in U \,| \,(\rad A)^n\cdot u=0\}$ by \ref{A/RadA semisimple}.
Given a projective indecomposable module $P$, the head $P/\rad P$ is the top radical layer of $P$ and $\soc P$ is the bottom radical layer of $P$.
\begin{figure}[H]\centering
$\begin{matrix}
P/\Rad P \\
*\\
\Soc P
\end{matrix}\;\;$
\caption{The radical series structure of $P$}
\end{figure} If $P$ is a projective indecomposable $kG$-module, $\soc P$ has the following property,  \cite[Corollary 8.5.3.(2)]{Webb_2016}.

\begin{Theorem}\label{projindeckGmodsimplesocle}
Let $k$ be a field.
Each indecomposable projective $kG$-module has a simple socle.
\end{Theorem}

\begin{proof}

If $P$ is an indecomposable projective module and $S$ is simple, $P^*$ is also an indecomposable projective module as homomorphisms $S \rightarrow P$ biject with homomorphisms $P^* \rightarrow S^*$ via duality. By \ref{multiplicitySofUisdimhomPStoU},
\begin{displaymath}
  \dim_k\Hom_{kG}(S, P)=\dim_k\Hom_{kG}(P^*, S^*)= \left\{
    \begin{array}{lr}
       \dim_k\End(S^*) & \;\text{if } P^*=P_{S^*}\\
      0 & \;\text{otherwise}.
    \end{array}
  \right.
\end{displaymath}

As $ \dim_k\Hom_{kG}(S, P)=\dim_k\End(S)$ by $\dim_k\End(S^*)=\dim_k\End(S)$, $P$ has the unique simple submodule $S$. It follows that $\Soc P$ is simple.
\end{proof}

The theorem \ref{pradpisosocp}, \cite[Theorem 8.5.5]{Webb_2016}, gives you the partial structure of the radical series for a projective indecomposable $kG$-module. That is, the theorem tells that the top layer and bottom layer of $P$ are the same (up to isomorphism) in a group algebra. Before that, the definition of symmetric algebra and the following lemma are introduced as preliminaries. \index{symmetric algebra}A finite-dimensional algebra $A$ over a field $k$ is called a \textit{symmetric algebra} if there exists a nondegenerate bilinear form $(\;\;,\;\;): A\times A \rightarrow k$ satisfying $(a,b)=(b,a)$ and $(ab, c)=(a, bc)$ for all $a, b, c\in A$.

\begin{Remark}
$kG$ is a symmetric algebra with the bilinear form defined on the basis elements by
\begin{displaymath}
  (g, h) = \left\{
    \begin{array}{lr}
      1 & \;gh=1\\
      0 & \;\text{otherwise.}
    \end{array}
  \right.
\end{displaymath} 
\end{Remark}

\begin{Lemma}\label{dimidempotent}
Let $A$ be a finite-dimensional algebra and $U$ be a finite-dimensional $A$-module.
If $e\in A$ is an idempotent, then $\dim_k\Hom_A(Ae,U)=\dim_k eU$.
\end{Lemma}

\begin{proof}
It suffices to prove that a map of vector spaces $\Psi: \Hom_A(Ae, U) \rightarrow eU$ defined by $\Psi(\phi)=\phi(e)$ is an isomorphism. $\Psi$ is well-defined as $\phi(e)=\phi(ee)=e\phi(e) \in eU$ for all $\phi \in \Hom_A(Ae,U)$. Let $\phi: Ae \rightarrow U$ be an $A$-module homomorphism such that $\phi(e)=0$. For any $ae\in Ae$, $\phi(ae)=a\phi(e)=0$, so $\phi=0$. This shows that $\Psi$ is injective. For surjectivity, let $b\in eU$. Define a map $\phi: Ae \rightarrow U$ by $\phi(ae)=ab$. Then $\phi(a_1e+a_2e)=\phi(a_1)b+\phi(a_2)b$. For any $r\in A$, $\phi(rae)=ra\cdot b=r\phi(ae)$. This shows $\phi\in\Hom_A(Ae,U)$ exists for any $b\in eU$.
\end{proof}

\begin{Theorem}\label{pradpisosocp}
Let $P$ be an indecomposable projective module for a group algebra $kG$.
Then $P/\Rad P \cong \Soc P$.

\end{Theorem}

\begin{proof}
By \ref{Onetoonecorrespondenceprojectivesimple}, there exists a primitive idempotent $e\in kG$ such that $P\cong kGe$ as $kG$-modules. The claim is that $\Soc kGe=\Soc kG\cdot e$. As $\Soc kG$ is an ideal, $\Soc kG\cdot e \subseteq \Soc kG$. Clearly, $\Soc kG\cdot e \subseteq kG\cdot e$. It follows that $\Soc kG\cdot e \subseteq kGe \cap \Soc kG =\Soc kGe$. $kGe \cap\Soc kG =\Soc kGe$ because $\Soc kGe$ is the largest semisimple submodule of $kGe$. Conversely, $\Soc kGe \subseteq \Soc kG$ as $\Soc kGe$ is semisimple. As $\Soc kGe$ is an ideal, $\Soc  kGe=\Soc kGe\cdot e \subseteq \Soc kG\cdot e$ so we've proved the claim. By \ref{dimidempotent}, \begin{equation}\label{dimkGeSockGeeSockGe} \dim_k\Hom_{kG}(kGe, \Soc kG\cdot e)=\dim_ke\Soc kG\cdot e.\end{equation} By \ref{projindeckGmodsimplesocle} and $P\cong kGe$, $\Soc kGe$ is simple. By the above equation \ref{dimkGeSockGeeSockGe} and Nakayama's Lemma \ref{nakayamalemma}, $e\Soc kG\cdot e \neq 0$ if and only if $\Soc kGe  \cong kGe/\Rad kGe=P/\rad P$. It remains to show $e\Soc kG\cdot e \neq 0$. If $e\Soc kG\cdot e =0$, then $0=(1, e\Soc kG\cdot e)=(e, \Soc kG\cdot e)=(\Soc kG\cdot e, e)=(kG\cdot \Soc kG\cdot e, e)=(kG, \Soc kG\cdot e\cdot e)=(kG, \Soc kGe)$. Since the bilinear form is nondegenerate, $\Soc kGe=0$, which leads to a contradiction.
\end{proof}

\begin{Corollary}\label{dialprojSisprojdualS}
For every simple $kG$-module $S$, $({P_S})^*\cong P_{S^*}$.
\end{Corollary}
\begin{proof}
In the proof of \ref{projindeckGmodsimplesocle}, ${P_S}^*$ is the projective cover of ${(\soc P_S)}^*$. By \ref{pradpisosocp}, $P_{S^*}=S^*$.
\end{proof}

$I_S=P_S$ for any simple modules $S$ by the following. By duality, $U \rightarrow I_U$ is the injective envelope of $U$ if and only if ${I_U}^* \rightarrow U^*$ is the projective cover of $U^*$. The morphism $S^*\rightarrow {(P_S)}^*$ is the injective envelope of $S^*$ by duality, so ${(P_S)}^*=I_S$. By \ref{dialprojSisprojdualS}, the result follows, as desired. Thus, for any projective indecomposable module $P$,
\begin{equation}\label{PPsIs}
P=P_S=I_S
\end{equation}
for some simple $S$. By \ref{pradpisosocp}, $P/\Rad P \cong \Soc P \cong S$ for some simple module $S$.
The radical series of $P$ is of the form
\begin{figure}[H]\label{radicalseriesform}
\centering
$P=\quad\begin{matrix}
S \\
*\\
S
\end{matrix}.$
\caption{The radical series of $P$}
\end{figure}

\chapter{The Klein Four-Group Algebra over Characteristic $2$}\chaptermark{The Klein Four-Group Algebra}
With the above constructions, we will see decomposition of group algebras. It starts with the Klein group algebra.

\section{A $p$-Group Algebra over Characteristic $p$}\chaptermark{p-Group Algebra}

\index{cyclic module}
 A module with a unique composition series is said to be \textit{uniserial}\index{uniserial}. That is, submodules of a uniserial module are linearly ordered by inclusion. A uniserial module is indecomposable.

\subparagraph{A Cyclic $p$-group Algebra over Characteristic $p$}
The following lemma \cite[Proposition 6.1.1]{Webb_2016} shows the structure of a cyclic $p$-group algebra.

\begin{Lemma}\label{kGkXXpn}
Let $k$ be a field of characteristic $p$ and $G$ be a cyclic $p$-group with order $p^n$.
Then there exists a ring isomorphism $kG \rightarrow k[X]/(X^{p^n})$ where $k(X)$ is the polynomial ring in an indeterminate $X$.
\end{Lemma}

\begin{proof}
Define a map $G\rightarrow k[x]/(X^{p^n})$ by $g^s \mapsto (X+1)^s$ for any integer $s$. We want to show that the map is isomorphism. $(X+1)^{p^n}=X^{p^n}+p(X^{P^n-1}+\cdots+X)+1\equiv 1 \mod X^{p^n}$. This implies that the map is a group homomorphism where $G$ is mapped onto a group generated by $X+1$. $1$ is mapped onto $1$, as $g^{p^n}=1$ is mapped onto $(X+1)^{p^n}\equiv 1 \mod X^{p^n}$. $g, g^2,\cdots, g^{p^n-1}$ are mapped onto $X+1, (X+1)^2=X^2+2X+1, \cdots, (X+1)^{p^n-1}$, respectively. $\{X+1, (X+1)^2, \cdots, (X+1)^{p^n-1}\}$ is linearly independent, as they have different powers to each other. Clearly, the set spans $k[x]/(X^{p^n})$. As the basis $\{1,g,\cdots, g^{p^n}\}$ of $kG$ bijects with the basis $\{X+1, (X+1)^2, \cdots, (X+1)^{p^n-1}\}$ of $k[x]/(X^{p^n})$, $kG\cong k[x]/(X^{p^n})$.
\end{proof}

For decomposing an arbitrary module over a field of characteristic $p$, the next theorem \cite[Theorem 6.1.2]{Webb_2016} is introduced. A module over a ring is said to be \textit{cyclic} if it is generated by one element. 

\begin{Theorem}
Let $k$ be a field of characteristic $p$.

\begin{enumerate}
\item
Every finitely generated $k[X]/(X^{p^n})$-module is a direct sum of cyclic modules $U_r=k[X]/(X^r)$ where $1 \leq r \leq p^n$.

\item
The $1$-dimensional module $U_1$ is the only simple module.

\item
Each module $U_r$ has a unique composition series.

\item
$U_r$ is indecomposable.

\item When $G$ is cyclic of order $p^n$, the regular module $kG$ has exactly $p^n$ indecomposable modules.
One of each dimension is $i$ with $1\leq i\leq p^n$. Each module has a unique composition series.
\end{enumerate}
\end{Theorem}
\begin{proof}
\begin{itemize}
\item[(i)]
We want to identify the modules for $k[X]/(X^{p^n})$ with the modules for $k[X]$ on which $X^{p^n}$ acts as zero. Every finitely generated $k[X]$-module is a direct sum of modules $k[X]/I$ where $I$ is an ideal. Then every $k[X]/(X^{p^n})$-module is a direct sum of modules $k[X]/I$ on which $X^{p^n}$ acts as zero. This is equivalent to $(X^{p^n})\subseteq I$. Let $a$ be an element such that $I=(a)$ where $a$ divides $X^{p^n}$. Then $I=(X^r)$ where $1\leq r\leq p^n$, so $k[X]/I=U_r$.
The submodules of $U_r$ must have the form $J/(X^r)$ for some ideal $J$ such that $(X^r)\subseteq J$.
\item[(iii)]
$U_i$ are precisely the submodules in the chain
\begin{equation*}
0\subset (X^{r-1})/(X^r)\subset (X^{r-2}/(X^r)\subset\cdots\subset (X)/(X^r)\subset U_r.
\end{equation*}
The chain is a composition series as $(X^{i-1})/(X^r)/(X^i)/(X^r)$ is of dimension 1 for all $1\leq i \leq r$, the chain is unique, and it is the complete list of $U_r$.
\item[(iv)]
For contradiction, suppose that $U_r=V\oplus W$ for some $V$ and $W$ as a nontrivial direct sum. Then it leads to a contradiction as $U_r$ has at least 2 composition series, $0\subset U_r/V\subset U_r/W\subset U_r$ or $0\subset U_r/W\subset U_r/V\subset U_r$.
\item[(ii)]
The only $U_r$ that is simple is $U_1$, the trivial module.
\item[(v)]
Let $G$ be a cyclic group of order $p^n$. Then $kG\cong k[X]/(X^{p^n})$ by \ref{kGkXXpn}. $kG$ has $p^n$ indecomposable modules, for each $1\leq i\leq p^n$, $\dim U_i=i$, and the corresponding composition series is unique.
\end{itemize}
\end{proof}

In a cyclic $p$-group algebra $kG$ over characteristic $p$, its indecomposable modules are all cyclic and uniserial.

\begin{Example}\label{c2c4c5}
Consider the cyclic $2$-group $C_2 = \langle g\,|\,g^2=1\rangle$. Over characteristic $2$, $kC_2 \cong k[X]/\langle X^2\rangle$.
$kC_2$ has a basis $\{1, \alpha\}$ where $\alpha=1-g$. $\alpha$ is nilpotent as $\alpha^2=(1-g)^2=0$. It has a unique composition series $\begin{smallmatrix} &1& \\ &\downarrow &\times \alpha \\ &\alpha &\end{smallmatrix}$.

$kC_4$ with a basis $\{1, \beta, \beta^2, \beta^3\}$ where $\beta=1-g_{C_4}$ over characteristic $2$ and $kC_5$ with a basis $\{1, \gamma, \gamma^2, \gamma^3, \gamma^4\}$ where $\gamma=1-g_{C_5}$ over characteristic $5$ have a unique composition seres $$\begin{matrix}\begin{smallmatrix} 1\\ \downarrow \\  \beta \\ \downarrow \\  \beta^2\\  \downarrow \\  \beta^3 \end{smallmatrix} &\quad \text{and}\quad& \begin{smallmatrix} 1 \\ \downarrow \\   \gamma\\  \downarrow \\ \gamma^2\\   \downarrow \\ \gamma^3\\   \downarrow \\ \gamma^4  \end{smallmatrix} \end{matrix},$$
respectively.

 \end{Example}

\subsection{A Simple $p$-Group Algebra Module}

\paragraph{Restriction on a Simple Module to a $p$-Group over Characteristic $p$}

\begin{Lemma}\label{sumofsimplesissemisimple}
Let $A$ be a ring with an identity $1$. Suppose that an $A$-module $U=S_1+\cdots+S_n$ can be written as the sum of finitely many simple modules $S_1,\cdots, S_n$. Then $U$ is semisimple.
\end{Lemma}
\begin{proof}
Let $V$ be any submodule of $U$. Choose a subset $I$ of $\{1,\cdots,n\}$ that is maximal subject to the condition that the sum $W=V\oplus(\oplus_{i\in I}S_i)$ is a direct sum. If $I=\emptyset$, then $W=V$. Assume that $I$ is nonempty maximal subset the set with the condition. We want to prove a claim that $W=U$. If $W\neq U$, then there exists a simple $S_j$ such that $S_j\subsetneq W$ for some $j$. As $S_j\cap W$ is a proper submodule of the simple module $S_j$, $S_j\cap W=0$. $S_j+W=S_j\oplus W$, which leads to a contradiction by the maximality of $I$. This implies that $U$ is the direct sum of $V$ and simples $S_i$ for $i\in I$. By taking $V=0$, $U$ is a direct sum of some simple modules, which proves that $U$ is semisimple.
\end{proof}

The weak form of Clifford's Theorem \cite[Theorem 5.3.1]{Webb_2016} is introduced to prove that a $p$-group algebra over characteristic $p$ has the only simple module, $k_G$.

\begin{Proposition}[Weak form of Clifford's Theorem]\label{Weak form of Clifford's Theorem}
Let $k$ be any field and $S$ be a simple $kG$-module. If $N$ is a normal subgroup of $G$, then the restriction $(S)_N$ is semisimple as a $kN$-module.
\end{Proposition}

\begin{proof}
Let $T$ be any simple $kN$-module of $(S)_N$. For any $g\in G$, a module $gT$ is a $kN$-submodule of $V$. When acting any $n\in N$ on the module $gT$, $n(gt)=g(g^{-1}ng)t\in gT$ as $g^{-1}ng \in N$ and $(g^{-1}ng)t \in T$. $gT$ is simple as if $W$ is a $kN$-submodule of $gT$, then $g^{-1}W$ is a submodule of the simple module $T$. $\sum_{g\in G}gT$ is a nonzero $G$-invariant subspace of the simple $kG$-module $S$, so $\sum_{g\in G}gT=S$. This implies that the $kN$-module $(S)_N$ is a sum of simple submodules. By \ref{sumofsimplesissemisimple}, $(S)_N$ is semisimple.
\end{proof}
The next proposition \cite[Proposition 6.2.1]{Webb_2016} shows that when $kG$ is a $p$-group algebra over characteristic $p$, $k_G$ is the only simple $kG$-module.
\begin{Proposition}\label{pgroupandtrivialmodule}
Let $k$ be a field of characteristic $p$ and $G$ be a $p$-group. Then there exists a unique simple $kG$-module, that is, the trivial module $k_G$.
\end{Proposition}

\begin{proof}
We use induction on the order of a $p$-group $G$. When $|G|=1$, i.e, $G$ is the identity group, $k_G$ is the only $kG$-module. Assume that for any $p$-group $H$ of smaller order, the only simple $kH$-module is the trivial module $k_H$. The order of $G$ is a power of $p$, so there exists a normal subgroup $N$ of $G$ with $[G:N]=p$. Let $S$ be any simple $kG$-module. Then, by the weak form of Clifford's Theorem, $(S)_N$ is semisimple. By the assumption, $N$ acts trivially on $S$, so $S$ is a k$G/N$-module. Again by the assumption, the only simple module of $G/N$ is the trivial module, so $S=k_{G/N}$.
\end{proof}

\begin{Example}\label{kC2C2R}
\begin{enumerate}
\item
From \ref{c2c4c5}, the radical series of $kC_2$, $kC_4$, and $kC_5$ are \begin{figure}[H]\centering$\begin{matrix}\begin{matrix} k_{C_2} \\ k_{C_2}\end{matrix}, \quad& \begin{matrix} k_{C_4}\\  k_{C_4}\\  k_{C_4}\\  k_{C_4} \end{matrix} \quad& \text{and} \quad & \begin{matrix} k_{C_5}\\  k_{C_5}\\  k_{C_5}\\  k_{C_5} \\ k_{C_5} \end{matrix} \end{matrix},$\end{figure}
respectively.

\item 
The Klein four-group algebra $kC_2\times C_2$ over characteristic $2$ has the trivial module $k_{C_2\times C_2}$ as the only simple module by \ref{pgroupandtrivialmodule}. This implies it has the only one projective indecomposable module $R=P_{k_{C_2\times C_2}}$. 
\end{enumerate}
\end{Example}

\subsection{The Projective Indecomposable Module of A $p$-Group}
The following theorem \cite[Theorem 8.1.1]{Webb_2016} shows that a $p$-group algebra $kG$ over characteristic $p$ is itself an indecomposable projective module.

\begin{Theorem}\label{pgroupprojindmod}\label{pgroupprojindmod2}
Let $k$ be a field of characteristic $p$ and $G$ be a $p$-group. Then
\begin{enumerate}
\item
The group algebra $kG$ is itself an indecomposable projective module which is the projective cover of $k_G$.

\item
Every finitely generated projective module is free.

\item
The only idempotents in $kG$ are $0$ and $1$.

\end{enumerate}

\end{Theorem}

\begin{proof}

\begin{enumerate}
\item
As any group algebra has its group as a basis, $kG$ is free. $kG$ has the only simple module, $k_G$. $kG$ has a unique projective indecomposable module $P_{k_G}$. By \ref{regmodAdecompassumofprojcovers}, $kG=P_{k_G}$ with the multiplicity one as $k_G$ is of dimension one. This shows that $kG$ is itself a projective indecomposable module.
\item
Every finitely generated projective module is a direct sum of indecomposable projective modules, which are direct summands of free modules.
\item
 Let $e$ be an idempotent of $kG$. Every idempotent $e\in kG$ decomposes $kG=kGe \oplus kG(1-e)$. Suppose $xe=y(1-e)$ for some $x,y \in kG$. Then $xe=xe^2=y(1-e)e=0$. This shows $kGe\, \cap \, kG(1-e)=0$. As $1=e+(1-e)\in kGe + kG(1-e)$, $kG= kGe+kG(1-e)$. As $kG$ is indecomposable, if $e \neq 0$, then $kG=kGe$. So $kG(1-e)=0$ and $e=1$.\end{enumerate}
\end{proof}
By the above theorem, any $p$-group algebra $kG$ is itself a projective indecomposable module.
\begin{figure}[h]\centering$kG=P_{k_G}.$\caption{A $p$-group algebra decomposition over characteristic $p$}\end{figure}

\subparagraph{The Decomposition of $kC_2\times C_2$ into Projective Indecomposable Modules}
By \ref{pgroupprojindmod}, $kC_2\times C_2$ over characteristic 2 is the projective indecomposable module $R=P_{k_{C_2\times C_2}}$.$$kC_2\times C_2=R=P_{k_{C_2\times C_2}}.$$
The decomposition of $kC_2\times C_2$ as a direct sum of projective indecomposable $kC_2\times C_2$-modules is $kC_2\times C_2=P_{k_{C_2\times C_2}}.$
The structure of $kC_2\times C_2$ is the radical series of $R$, by the diagram \ref{radicalseriesform}, which is of the form \begin{figure}[H]\centering$\begin{matrix} k_{C_2\times C_2} \\ * \\ k_{C_2\times C_2} \end{matrix}$.\end{figure} As the dimension of $kC_2\times C_2$ is 4, $R$ is of dimension four. This implies that the radical series of $R$ has the structure either
\begin{figure}[h]\centering$$\begin{matrix}\begin{matrix}  & k_{C_2\times C_2} &  \\ k_{C_2\times C_2} & &k_{C_2\times C_2} \\  & k_{C_2\times C_2} &  \end{matrix} &\quad\text{or} \quad\quad&\begin{matrix} k_{C_2\times C_2}  \\ k_{C_2\times C_2}  \\ k_{C_2\times C_2}  \\ k_{C_2\times C_2}  \end{matrix}\end{matrix}\;\;.$$\caption{Possible structures of $R$}\end{figure}

Let $X=1-x$ and $Y=1-y$ where $x$ and $y$ are the generators of the Klein four-group. $X$ and $Y$ are nilpotent by $X^2=Y^2=0$. $kC_2\times C_2 $ has two different ideals, one generated by $X$ and the other by $Y$. This implies there are at least two maximal ideals in $kC_2\times C_2 $. The structure of $kC_2\times C_2$ is $k[X,Y]/\langle X^2, Y^2 \rangle$.\begin{figure}[H]
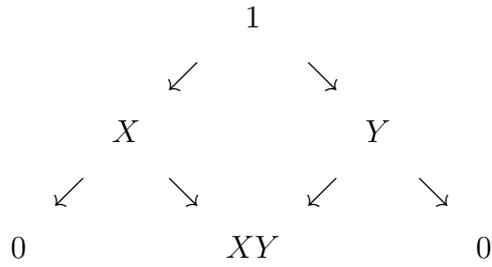

\centering$\begin{matrix} & && & 1 && && \\   &&&   \swarrow &&  \searrow & &  & \\   &&X  &&  && Y& &\\ & \swarrow &  & \searrow  &   &\swarrow & & \searrow & \\  0&&&& XY& &&&0  \end{matrix}.$\caption{$k[X,Y]/\langle X^2, Y^2\rangle$}\end{figure}

$\rad kC_2\times C_2=k\langle X,Y\rangle=\begin{smallmatrix} k_{C_2\times C_2} & \oplus &k_{C_2\times C_2} \\  & k_{C_2\times C_2} &  \end{smallmatrix}$. The radical series of $kC_2\times C_2$ is 
\begin{figure}[H]\centering$\begin{matrix}  & k_{C_2\times C_2} &  \\ k_{C_2\times C_2} & \oplus  &k_{C_2\times C_2} \\  & k_{C_2\times C_2} &  \end{matrix}.$\caption{The radical series of $kC_2\times C_2$}\end{figure}
The decomposition of $kC_2\times C_2$ into projective indecomposable modules is itself.

\subparagraph{The Cartan Matrix of $kC_2\times C_2$}
The Cartan matrix of $kC_2\times C_2$ is in $M_1(\mathbb{N}_0)$ as the group algebra has the only simple module $k_{C_2\times C_2}$. By counting the number of the module in $R$, the Cartan matrix of $kC_2\times C_2$ is
\begin{figure}[h]\centering$[4]$.\caption{The Cartan matrix of $kC_2\times C_2$}\end{figure}

\subparagraph{The decomposition of $kC_2\times C_2$ into Blocks}
As $kC_2\times C_2=P_{k_{C_2\times C_2}}$ and by \ref{regmoddecompwithidempots}, the identity $1$ is the only primitive idempotent. This shows that $kC_2\times C_2$ has the only one block, which is $kC_2\times C_2$ itself.

\chapter{The Alternating Four-Group Algebra over Characteristic $2$}\chaptermark{Group Algebra $kA_4$}

\section{Simple $kA_4$-modules}

The Klein four-group is isomorphic to a Sylow $p$-group of the alternating four-group $A_4$. A \textit{Sylow $p$-subgroup} of a finite group $G$ is a maximal subgroup whose order is the highest power of a prime $p$ which divides the order of $G$. \label{SylowpsubgroupG_p}\index{Sylow $p$-subgroup} Every finite group $G$ has a Sylow $p$-subgroup by Sylow Theorems.
For any $n \in \mathbb{N}$, define $n_p$\index{$n_p$} as the largest power of a prime $p$ that divides $n$. Explicitly, if $n$ can be factorised as $p_1^{a_1}\cdots p_r^{a_r}$, then $n_{p_i}=p_i^{a_i}$. We define $n_{q}$ as $1$ for $q \nmid n$.
$|G|_p$ is the order of a Sylow $p$-subgroup $H$ of $G$. $|A_4|_2$ and $|A_5|_2$ are $4$.

Consider $V_4=\langle (1,2)(3,4), (1,3)(2,4)\rangle$ in $A_4$. Then $V_4$ is isomorphic to the Klein four-group as the generators are of order two and commutative. The order of $V_4$ is four, so $V_4$ is a Sylow $2$-subgroup of $A_4$.
As $V_4$ is a normal subgroup of $A_4$, $A_4/V_4$ forms a group of order $3$ by $|A_4/V_4|=[A_4:V_4]=3$. As the cyclic group $C_3$ is the unique group of order $3$, $A_4/V_4$ is isomorphic to $C_3$. If we denote $O_p(G)$\index{$O_p(G)$} as the unique largest normal $p$-subgroup of $G$, then $V_4=O_2(A_4)$. \index{semidirect} A group $G$ has a \textit{semidirect} product if $G = N \rtimes K$ for a normal subgroup $N$ and a subgroup $K$ such that $N\cap K=1$. Then $A_4 = V_4 \rtimes \langle (1,2,3)\rangle \cong C_2\times C_2 \rtimes C_3$. The next lemma \cite[Corollary 6.2.2]{Webb_2016} shows that $O_p(G)$ acts trivially on each simple $kG$-modules.

\begin{Lemma}\label{OpGtriviallyactsonsimples}
Let $k$ be a field of characteristic $p$ and $G$ be a finite group. Then the common kernel of the action of $G$ on all the simple $kG$-modules, which is $\{ g\in G \,|\, gs=s \text{ for all simples $S$ and $s\in S$}\}$, is $O_p(G)$.
\end{Lemma}

\begin{proof}
Let $H$ be the kernel of the $G$-action on all simple $kG$-modules. That is,
$$H=\{g\in G \,|\, gs=s, \text{ for all simples $S$ and $s\in S$}\}.$$
By Clifford's Theorem, if $S$ is a simple $kG$-module, then $(S)_{O_p(G)}$ is semisimple. By \ref{pgroupandtrivialmodule}, $O_p(G)$ acts trivially on $S$, so $O_p(G)\subseteq H$. Conversely, it suffices to show that $H$ contains no element of order prime to $p$. Assume that there is such an element $h\in H$ with $(\ord(h),p)=1$. Then $(kG)_{\langle h\rangle}$ is semisimple $k\langle h \rangle$-module, so is the direct sum of modules $(S)_{\langle h\rangle}$. As $h$ acts trivially on all of the direct summands $(S)_{\langle h\rangle}$, $h$ acts trivially on $kG$, which leads to contradiction. This shows that $H$ is a $p$-subgroup. Clearly, $H$ is a normal subgroup of $G$. It follows $H\subseteq O_p(G)$.
\end{proof}

By \ref{OpGtriviallyactsonsimples}, the simple $kG$-modules are precisely the simple $k[G/O_p(G)]$-modules. It follows from the quotient homomorphism $G\rightarrow G/O_p(G)$. $A_4/V_4\cong C_3$. For an abelian group, every simple has dimension 1 \cite[Theorem 5.3.3]{Webb_2016}, so it's easier to find simple modules for cyclic groups. From the simples of $kC_3$, we want to find all simple $kA_4$-modules.
\begin{Proposition}\label{abeliansimplekGmoddim1}
Let $k$ be an algebraically closed field and $G$ be an abelian group. Then every simple $kG$-module has its dimension $1$.
\end{Proposition}

\begin{proof}
Let $S$ be a simple $kG$-module.
Every $g$ action on $S$ has an eigenvalue $\lambda$ in the algebraically closed field $k$, so the corresponding eigenspace $S_{\lambda}$ is nonzero. We want to show $hS_{\lambda}=S_{\lambda}$ for every $h$ in $G$. Each $v$ in $S_{\lambda}$ satisfies $gv=\lambda v$.  $hv\in S_{\lambda}$, as $g(hv)=h(gv)=h(\lambda v)=\lambda hv$, so $v\in hS_{\lambda}$. For any $w=hv'\in hS_{\lambda}$, $g(hv')=h(gv')=h(\lambda v')=\lambda(hv')$, so $w\in S_{\lambda}$. This shows $S_{\lambda}$ is a $kG$-submodule of $S$ and $S_{\lambda}=S$ as $S$ is simple. This means the action of every element $g$ in $G$ is a scalar multiplication on $S$. This implies $S$ has dimension $1$.
\end{proof}

By the above theorem, every simple $kC_3$-module has dimension one. In representation theory, we know that $C_3$ has three $1$-dimensional irreducible representations. We use a primitive third root of unity to explicitly write all irreducible representations of $C_3$. Consider the polynomial equation $x^2+x+1=0$ in $\mathbb{F}_2[x]$. As any of zero and one is not a root of the equation, the polynomial is irreducible over $\mathbb{F}_2$. Define $\omega$ as a root of this polynomial. Then $\mathbb{F}_4=\{0, 1, \omega, \omega+1\}$ forms a field, as $\omega$ and $\omega+1$ are inverses to each other by $\omega(\omega+1)=1$. $\omega$ is a primitive $3$rd root of unity by $\omega^3=\omega^2\times \omega=(\omega+1)\omega=1$ and $\omega^2\neq 0$. The irreducible representations of $C_3$ where $g$ is a generator are defined as follows:
\begin{align*}
k_{C_3}=T_1':& \;\;C_3 \rightarrow \GL_1(k) \text{ defined by } g \mapsto (1),\\
T_2': &\;\;C_3 \rightarrow \GL_1(k) \text{ defined by } g \mapsto (\omega),\;\;\text{and}\\
T_3': &\;\;C_3 \rightarrow \GL_1(k) \text{ defined by } g \mapsto (\omega^2).
\end{align*}
By identifying $A_4/V_4$ with $C_3$, $k(A_4/V_4)$ has $3$ simple modules $T_i'$ of dimension 1 where $g$ is replaced by the coset $(1,2,3)V_4$ in $A_4/V_4$. We want to inflate the simples $T_i'$ from $A_4/V_4$ to $A_4$. Let $N$ be a normal subgroup of $G$ and $f:N\rightarrow G/N$ be a quotient map defined by $f(g)=gN$. The \textit{inflation} $\Inf_{G/N}^G(U)$ on the module $U$ from $G/N$ to $G$ is a module with the action defined by $g\cdot u=(gN)u$ for all $g\in G$ and $u\in U$. By \ref{OpGtriviallyactsonsimples}, $V_4$ acts trivially on all simple modules $S$ in $kA_4$. This implies that the simple $kA_4$-modules $T_i$ are precisely the inflated simple modules $\Inf_{A_4/V_4}^{A_4}(T_i')$ from $A_4/V_4$ to $A_4$, $T_i=\Inf_{A_4/V_4}^{A_4}(T_i')$. The simple $kA_4$-modules $T_i$ are
\begin{align*}
k_{A_4}=T_1: &\;\;A_4 \rightarrow GL_1(k) \text{ where } (1, 2, 3) \text{ maps onto } (1),&\\
T_2: &\;\;A_4 \rightarrow \GL_1(k) \text{ where } (1, 2, 3) \text{ maps onto } (\omega),\;\;\text{and}\\
T_3: &\;\;A_4 \rightarrow \GL_1(k) \text{ where } (1, 2, 3) \text{ maps onto } (\omega^2).&
\end{align*}

\section{Indecomposable $kA_4$-modules}\label{IndecompM_iandPM_iIM_i}

The following $2$-dimensional $kA_4$-modules are given in the notes \cite{C.W.Eaton}.
$$M_1: \begin{smallmatrix} (1,2,3)\mapsto \big[\begin{smallmatrix} \omega & 0 \\0 &\omega^2 \end{smallmatrix}\big] & (1,2)(3,4) \mapsto \big[\begin{smallmatrix} 1 & \omega \\0 &1 \end{smallmatrix}\big] \end{smallmatrix}, \quad M_2: \begin{smallmatrix} (1,2,3)\mapsto \big[\begin{smallmatrix} \omega^2 & 0 \\0 &\omega \end{smallmatrix}\big] &(1,2)(3,4) \mapsto \big[\begin{smallmatrix} 1 & \omega \\0 & 1 \end{smallmatrix}\big] \end{smallmatrix},$$
$$M_3: \begin{smallmatrix} (1,2,3)\mapsto \big[\begin{smallmatrix} 1 & 0 \\0 & \omega \end{smallmatrix}\big] & (1,2)(3,4) \mapsto \big[\begin{smallmatrix} 1 & \omega \\0 & 1 \end{smallmatrix}\big] \end{smallmatrix}, \quad M_4:\begin{smallmatrix} (1,2,3)\mapsto \big[\begin{smallmatrix} \omega & 0 \\0 & 1 \end{smallmatrix}\big] & (1,2)(3,4) \mapsto \big[\begin{smallmatrix} 1 & \omega \\0 & 1 \end{smallmatrix}\big] \end{smallmatrix},$$
$$M_5:\begin{smallmatrix}  (1,2,3)\mapsto \big[\begin{smallmatrix} {\omega}^2 & 0 \\0 & 1 \end{smallmatrix}\big] & (1,2)(3,4) \mapsto \big[\begin{smallmatrix} 1 & \omega \\0 & 1 \end{smallmatrix}\big] \end{smallmatrix}, \quad \text{and} \quad M_6: \begin{smallmatrix}  (1,2,3) \mapsto  \big[\begin{smallmatrix} 1 & 0 \\0 & {\omega}^2 \end{smallmatrix}\big]  &  (1,2)(3,4) \mapsto \big[\begin{smallmatrix} 1 & \omega \\0 & 1 \end{smallmatrix}\big]   \end{smallmatrix}.$$

\subparagraph{Radical Series of $M_i$}
$M_1$ has a submodule $N_1=\left\{\big[\begin{smallmatrix} x \\0 \end{smallmatrix}\big]: x\in k\right\}$ with a basis $\big\{\big[\begin{smallmatrix} 1 \\0 \end{smallmatrix}\big]\big\}$ which is isomorphic to $T_2$, as
$(1,2,3) \mapsto (\omega)$ by $(1, 2, 3)\cdot \big[\begin{smallmatrix} 1 \\0 \end{smallmatrix}\big]=\big[\begin{smallmatrix}  \omega & 0\\  0 & {\omega}^2 \end{smallmatrix}\big]\big[\begin{smallmatrix} 1 \\0 \end{smallmatrix}\big]=\omega\big[\begin{smallmatrix} 1 \\0 \end{smallmatrix}\big]$ and similarly, $(1,2)(3,4) \mapsto (1)$.
$M_1$ has the quotient $M_1/N_1=\left\{\big[\begin{smallmatrix} 0 \\ y \end{smallmatrix}\big]+N_1: y\in k \right\}$ with a basis $\big\{\big[\begin{smallmatrix} 0 \\ 1 \end{smallmatrix}\big]+N_1\big\}$ isomorphic to $T_3$, as $(1,2,3) \mapsto \left(\omega^2\right)$ by $(1, 2, 3) \cdot (\big[\begin{smallmatrix} 1 \\ 0 \end{smallmatrix} \big]+N_1)=\big[\begin{smallmatrix}  \omega & 0\\  0 & {\omega}^2 \end{smallmatrix}\big] \big[\begin{smallmatrix} 0 \\1 \end{smallmatrix}\big]+N_1={\omega^2} \left(\big[\begin{smallmatrix} 0 \\ 1 \end{smallmatrix}\big]+N_1\right)$ and $(1,2)(3,4) \mapsto \left(1\right)$. $T_3$ is the unique quotient of $M_1$ and $T_2$ is the unique submodule of $M_1$ as $M_1$ is of dimension two. Therefore, the radical series of $M_1$ is $\begin{smallmatrix} T_3= M_1/\Rad M_1 \\ T_2= \Rad M_1 \end{smallmatrix}$. The structure shows that $M_1$ is uniserial and thus indecomposable.
$$M_1: \begin{smallmatrix} T_3 \\ T_2 \end{smallmatrix}\;\;.$$

We do the same process on $M_2,\cdots, M_6$. $M_2$ has a submodule $N_2$ with a basis $\big\{\big[\begin{smallmatrix} 1 \\0 \end{smallmatrix}\big]\big\}$ which is isomorphic to $T_3$ by $(1, 2, 3)\cdot \big[\begin{smallmatrix} 1 \\0 \end{smallmatrix}\big]=\big[\begin{smallmatrix}  {\omega}^2 & 0\\  0 & \omega \end{smallmatrix}\big]\big[\begin{smallmatrix} 1 \\0 \end{smallmatrix}\big]={\omega}^2 \big[\begin{smallmatrix} 1 \\0 \end{smallmatrix}\big],$ i.e., $(1,2,3) \mapsto ({\omega}^2)$. The  quotient $M_2/N_2\cong T_2$ as $(1,2,3) \mapsto \left(\omega \right)$ by $(1, 2, 3) (\big[\begin{smallmatrix} 0 \\ 1 \end{smallmatrix} \big]+N_2)=\big[\begin{smallmatrix}  {\omega}^2 & 0\\  0 & \omega \end{smallmatrix}\big] \big[\begin{smallmatrix} 0 \\1 \end{smallmatrix}\big]+N_2=\omega \left(\big[\begin{smallmatrix} 0 \\ 1 \end{smallmatrix}\big]+N_2\right)$. The radical series of $M_2$ is $$M_2: \begin{smallmatrix} T_2 \\ T_3 \end{smallmatrix}.$$

A submodule $N_3$ of $M_3$ with a basis $\{\big[\begin{smallmatrix} 1 \\0 \end{smallmatrix}\big]\}$ is isomorphic to $T_1$ by $(1,2,3) \mapsto (1)$. The quotient $M_3/N_3\cong T_2$ as $(1,2,3) \mapsto \left(\omega \right)$. The radical series of $M_3$ is $$M_3: \begin{smallmatrix} T_2 \\ T_1 \end{smallmatrix}.$$

A submodule $N_4$ of $M_4$ with a basis $\{\big[\begin{smallmatrix} 1 \\0 \end{smallmatrix}\big]\}$ is isomorphic to $T_2$ by $(1,2,3) \mapsto (\omega)$. The quotient $M_4/N_4\cong T_1$ by $(1,2,3) \mapsto \left(1 \right)$. The radical series of $M_4$ is \begin{equation*}M_4: \begin{smallmatrix} T_1 \\ T_2 \end{smallmatrix}.\end{equation*}

A submodule $N_5$ of $M_5$ with a basis $\{\big[\begin{smallmatrix} 1 \\0 \end{smallmatrix}\big]\}$ is isomorphic to $T_3$ by $(1,2,3) \mapsto (\omega^2)$. The quotient $M_5/N_5\cong T_1$ by $(1,2,3) \mapsto \left(1 \right)$. The radical series of $M_5$ is \begin{equation*}M_5: \begin{smallmatrix} T_1 \\ T_3 \end{smallmatrix}.\end{equation*}

A submodule $N_6$ of $M_6$ with a basis $\{\big[\begin{smallmatrix} 1 \\0 \end{smallmatrix}\big]\}$ is isomorphic to $T_1$ by $(1,2,3) \mapsto (1)$. The quotient $M_6/N_6\cong T_3$ by $(1,2,3) \mapsto \left( \omega^2 \right)$. The radical series of $M_6$ is $$M_6: \begin{smallmatrix} T_3 \\ T_1 \end{smallmatrix}.$$
$M_i$ are uniserial by their radical series structures and thus indecomposable.
\begin{figure}[H]\label{M_i}
\centering
\begin{align}  M_1: \begin{smallmatrix} T_3 \\ T_2 \end{smallmatrix},\quad &\quad M_2: \begin{smallmatrix} T_2 \\ T_3 \end{smallmatrix}, & M_3: \begin{smallmatrix} T_2 \\ T_1 \end{smallmatrix}, \quad & \quad M_4: \begin{smallmatrix} T_1 \\ T_2 \end{smallmatrix}, & M_5: \begin{smallmatrix} T_1 \\ T_3 \end{smallmatrix}, &\quad \text{and}&M_6: \begin{smallmatrix} T_3 \\ T_1 \end{smallmatrix}.  \end{align}
\caption{Indecomposable $kA_4$-modules $M_i$}
\end{figure}

\subparagraph{The Projective Cover and Injective Envelope of $M_i$}
As $kA_4$ has three simples, by \ref{Onetoonecorrespondenceprojectivesimple}, it has three projective indecomposable modules, say $Q_i$. The radical series of $Q_i$ can be found from the projective covers and the injective envelopes of the indecomposable modules $M_i$. By \ref{M_i}, $M_i/\Rad M_i$ and $\soc M_i$ are simple. By the relation \ref{PPsIs},
\begin{equation*}
Q_i= P_{T_i}= I_{T_i} \quad \text{for all} \; i=1,2,3.
\end{equation*} As $P_{T_i}=Q_i$ and $P_{M_i} \rightarrow M_i$ is an essential epimorphism, $P_{M_i}$ is one of $Q_1$, $Q_2$ and $Q_3$. By the duality, so is $I_{M_i}$. Let $T_l=M_i/\rad M_i$ and $T_k=\soc M_i$ where $M_i:\begin{smallmatrix} T_l \\ T_k \end{smallmatrix}$. We recall that $P_{M_i} \rightarrow \begin{smallmatrix} T_l \\ T_k\end{smallmatrix}$ is an essential epimorphism and $\begin{smallmatrix} T_l \\ T_k \end{smallmatrix} \rightarrow I_{M_i}$ is an essential monomorphism.
Then $P_{M_i}=Q_l$ since, among $Q_i$, only $Q_l \rightarrow T_l$ maps $P_{M_i}$ onto $T_l\cong P_{M_i}/\rad P_{M_i}$. Similarly, $I_{M_i}=Q_k$ as, among $Q_i$, only $T_k\rightarrow Q_k$ maps $T_k \cong \soc P_{M_i}$ onto $Q_k$. For $M_i=\begin{smallmatrix} T_l \\ T_k \end{smallmatrix}$,
\begin{equation*}
P_{M_i}=Q_l, \quad\text{and}\quad I_{M_i}=Q_k.
\end{equation*}
We obtain the following result.
 \begin{align}\label{Qiprojcovinjcov}
 Q_1=P_{M_4}=P_{M_5}=I_{M_3}=I_{M_6}\\
Q_2=P_{M_2}=P_{M_3}=I_{M_1}=I_{M_4}\\
\quad Q_3=P_{M_1}=P_{M_6}=I_{M_1}=I_{M_5}.
\end{align}

\section{Projective Indecomposable $kA_4$-modules}\label{PIQi}
With the results in the previous section, we want to find the radical series of $Q_i$. First, we find the dimension of $Q_i$.
\subparagraph{The Dimension of $Q_i$}
In a group algebra over characteristic $p$, the dimension of a projective module is divisible by the largest power of $p$ that divides the group order \cite[Corollary 8.1.3]{Webb_2016}.

\begin{Corollary}\label{largpowerofpdivdesdimofprojmod}
Let $k$ be a field of characteristic $p$ and $p^a$ be the largest power of $p$ that divides $|G|$. If $P$ is a projective $kG$-module, then $p^a| \dim P$.

\end{Corollary}
\begin{proof} 
$p^a$ is the order of a Sylow $p$-subgroup $H$ of $G$ by \ref{SylowpsubgroupG_p}. If $P$ is a projective $kG$-module, then $(P)_H$ is projective by \ref{indrestrprojmodareproj}. $(P)_H$ is free as a $kH$-module by \ref{pgroupprojindmod2}, so $|H|$ divides $\dim (P)_H$.
\end{proof}

By the above corollary, each dimension of $Q_i$ is divisible by $4$, so $12 \leq \dim Q_1 + \dim Q_2 + \dim Q_3$. The equality holds as $12=\dim Q_1 + \dim Q_2 + \dim Q_3$ by \ref{dimAbydecompproj}. Each $\dim Q_i$ is of dimension 4. With the relation $Q_i/\Rad Q_i \cong \Soc Q_i \cong T_i$, each radical series of $Q_i$ has two possible structures,
\begin{figure}[h]
\centering
$\begin{matrix} \begin{matrix} T_i \\ \openbox \\ \openbox \\ T_i \end{matrix} $\quad\quad  \text{or} \quad\quad$\begin{matrix} &T_i& \\ \openbox & & \openbox \\  &T_i& \;\;.\end{matrix}\end{matrix}$
\caption{Possible radical series structures of $Q_i$}
\end{figure}

\paragraph{Projective Indecomposable Modules $Q_i$}
By the equation \ref{Qiprojcovinjcov}, $Q_1$ has quotients $\begin{smallmatrix} T_1 \\ T_2 \end{smallmatrix} $ and $\begin{smallmatrix} T_1 \\ T_3 \end{smallmatrix} $, and submodules $\begin{smallmatrix} T_2 \\ T_1 \end{smallmatrix}$ and $\begin{smallmatrix} T_3 \\ T_1 \end{smallmatrix}$. This implies $\Rad Q_1= \begin{smallmatrix} T_2 & & T_3 \\  &T_1& \end{smallmatrix}$, $\Rad^2 Q_1 = T_1$ and $T_2 \oplus T_3=\Rad Q_1/\Rad^2 Q_1$. Similarly, $Q_2$ has quotients $\begin{smallmatrix} T_2 \\ T_3 \end{smallmatrix} $ and $\begin{smallmatrix} T_2 \\ T_1 \end{smallmatrix}$, and submodules $\begin{smallmatrix} T_3 \\ T_2 \end{smallmatrix}$ and $\begin{smallmatrix} T_1 \\ T_2 \end{smallmatrix}$. $Q_3$ has quotients $\begin{smallmatrix} T_3 \\ T_2 \end{smallmatrix} $ and $\begin{smallmatrix} T_3 \\ T_1 \end{smallmatrix}$, and submodules $\begin{smallmatrix} T_2 \\ T_3 \end{smallmatrix}$ and $\begin{smallmatrix} T_1 \\ T_3 \end{smallmatrix}$. The radical series of $Q_i$ are
\begin{figure}[h]\centering
\begin{align}
\label{Q1projindradser}
Q_1: &\;\; \begin{matrix} &T_1& \\T_2 & \oplus & T_3 \\  &T_1& \;\;. \end{matrix}\;, &
Q_2:&\;\; \begin{matrix} &T_2& \\T_1 & \oplus & T_3 \\  &T_2& \end{matrix}\;,& Q_3:\;\;\begin{matrix} &T_3& \\T_1 &\oplus & T_2 \\  &T_3&\;\; \end{matrix}\;.
\end{align}
 \caption{The projective indecomposable modules of $A_4$}
\end{figure}

\subparagraph{Decomposition of $kA_4$ into Projective Indecomposable Modules}
The decomposition of $kA_4$ as a direct sum of projective indecomposable modules (up to isomorphism) is
 \begin{figure}[h]\label{pimA4}\centering
\begin{align*}
kA_4 \cong Q_1 \oplus Q_2 \oplus Q_3\\
\text{where}\quad & Q_1:\begin{matrix} &T_1& \\T_2 & & T_3 \\  &T_1& \end{matrix}, & Q_2:\begin{matrix} &T_2& \\T_1 & & T_3 \\  &T_2& \end{matrix}, \quad \text{ and }\;\; & Q_3:\begin{matrix} &T_3& \\T_1 & & T_2 \\  &T_3& \end{matrix}.
\end{align*}
\end{figure}
 
 Simply, $$kA_4 \cong {}^{T_1}T_2 T_3 {}_{T_1}\oplus {}^{T_2}T_1 T_3 {}_{T_2} \oplus {}^{T_3}T_1 T_2 {}_{T_3}.$$

\subparagraph{The Cartan Matrix of $kA_4$}
The Cartan matrix of $kA_4$ is in $M_3(\mathbb{N}_0)$ since $kA_4$ has three simple modules. As $P_{T_j}=Q_j$, $(i,j)$th entry of the matrix is obtained by finding the number of multiplicity of $T_i$ in $Q_j$ explicitly described in the diagram \ref{Q1projindradser}. Hence, the Cartan matrix $(C_{T_iT_j})$ of $kA_4$ is
\begin{figure}[h!]\label{CartanmatrixA4}
\centering
$\left[\begin{matrix}  2 & 1& 1\\ 1&2&1\\1&1&2       \end{matrix}\right].$
\caption{The Cartan matrix of $kA_4$}
\end{figure}

\subparagraph{Primitive Orthogonal Idempotents in $kA_4$} By \ref{regmoddecompwithidempots}, each $Q_i$ corresponds with a primitive orthogonal idempotent in $kA_4$. We want to find all the corresponding primitive orthogonal idempotents. The following theorem\cite[Theorem 3.6.2]{Webb_2016} gives a formula for block idempotents in a complex group algebra. We find the block idempotents in the complex group algebra $\mathbb{C}A_4$ and the elements modulo 2 on coefficients are the primitive orthogonal idempotents in $kA_4$.

\begin{Theorem}\label{centidemformula}
Let $\chi_1,\cdots,\chi_r$ be the irreducible complex characters of $G$ with degrees $d_1,\cdots, d_r$. ($\chi_i(1)=d_i$). The primitive central idempotent elements in $\mathbb{C}G$ are the elements
\begin{equation} \frac{d_i}{|G|}\sum_{g\in G}\chi_i(g^{-1})g
\end{equation}
where $1\leq i \leq r$.

The corresponding indecomposable ring summand of $\mathbb{C}G$ has an irreducible representation that affords the character $\chi_i$.
\end{Theorem}

\begin{proof}
Let $\rho_i$ be the representation which affords $\chi_i$. The algebra map $\rho_i : \mathbb{C}G\rightarrow M_{d_i}(\mathbb{C})$ is a projection onto the $i$th matrix summand in a decomposition of $\mathbb{C}G$. The $i$th matrix summand is a direct summand of $\mathbb{C}G$, as $\mathbb{C}G$ can be decomposed as a direct sum of matrix rings by Artin-Wedderburn Theorem. For any field $k$, the matrix ring $M_n(k)$ is indecomposable, so $Z(M_n(k))\cong k$ and the only nonzero idempotent in $k$ is $1$. The decomposition of $\mathbb{C}G$ as a direct sum of matrix rings is the unique decomposition of $\mathbb{C}G$ as a sum of indecomposable ring summands. Recall that a block idempotent is the identity in a direct summand and is zero in the others. The block idempotents $e_i$ in $\mathbb{C}G$ has the properties $\rho_i(e_i)=I$ and $\rho_j(e_i)=0$ if $i\neq j$. Each element $\frac{d_i}{|G|}\sum_{g\in G}\chi_i(g^{-1})g$ for $g\in G$ satisfies the properties so is a block idempotent, as 
\begin{align*}
\rho_j\left(\frac{d_i}{|G|}\sum_{g\in G}\chi_i(g^{-1})g\right)=& \frac{d_i}{|G|d_j}\sum_{g\in G}\chi_i(g^{-1})\chi_j(g)\cdot I \\
=&\frac{d_i}{d_j}\langle x_i, x_j\rangle\cdot I=\frac{d_i}{d_j}\delta_{i,j}\cdot I=\delta_{i,j}\cdot I.
\end{align*}
It follows that the elements $\frac{d_i}{|G|}\sum_{g\in G}\chi_i(g^{-1})g$ for $g\in G$ are the primitive central idempotent elements of $\mathbb{C}G$.
\end{proof}

Using \ref{centidemformula}, we want to find primitive orthogonal central idempotent elements of $kK$ where $K=\langle(1,2,3)\rangle$, a subgroup of $A_4$ with order $3$.
In $\mathbb{C}K$,
$$e_i=\frac{d_i}{|K|}\sum_{g\in K}\chi_i(g^{-1})g$$
where $1\leq i \leq r$.
The character table of $C_3$ is given in \ref{chartabofC3}. The primitive central orthogonal idempotents in $\mathbb{C}K$ are
\begin{align*}
{e_1}'=&1/3\big\{()+(1,2,3)+(1,3,2)\big\},\\
{e_2}'=&1/3\big\{()+e^{\frac{\pi i}{3}}(1,2,3)+e^{\frac{2 \pi i}{3}}(1,3,2)\big\},\\
{e_3}'=&1/3\big\{()+e^{\frac{2 \pi i}{3}}(1,2,3)+e^{\frac{\pi i}{3}}(1,3,2)\big\}.
\end{align*}
The modulo map $\mod 2: \mathbb{Z}[e^{\frac{\pi i}{3}}]\rightarrow k$ defined by $e^{\frac{\pi i}{3}}\mapsto \omega$ and $n\mapsto (n \mod 2)$ preserves some properties in $\mathbb{Z}[e^{\frac{\pi i}{3}}]$, so the central orthogonal idempotents in $kK$ are
\begin{align*}
{e_1}=&()+(1,2,3)+(1,3,2)\\
{e_2}=&()+\omega(1,2,3)+\omega^2(1,3,2)\\
{e_3}=&()+\omega^2(1,2,3)+\omega(1,3,2).
\end{align*}
$1=e_1+e_2+e_3$ where $e_i$ are orthogonal idempotents in $kA_4$. The decomposition of $kA_4$ as the direct sum of submodules corresponding to $1=e_1+e_2+e_3$ is
$$kA_4=kA_4e_1\oplus kA_4e_2\oplus kA_4e_3$$
where ${e_1}=()+(1,2,3)+(1,3,2)$, ${e_2}=()+\omega(1,2,3)+\omega^2(1,3,2)$, and ${e_3}=()+\omega^2(1,2,3)+\omega(1,3,2)$. As we already know $kA_4=Q_1\oplus Q_2\oplus Q_3$, $Q_i=kA_4e_i$.

\subparagraph{Decompositions of $kA_4$ into Blocks} 
From \ref{CartanmatrixA4}, the multiplicity of $T_i$ in $Q_j$ are nonzero for $1\leq i, j\leq 3$. By \ref{blocksandmultiplicity}, this implies that all $Q_i$ are in the same block. It follows that $kA_4$ has only one block, $kA_4$ itself. $kA_4$ is exactly the decomposition into a block.

\chapter{The Alternating 5-Group Algebra over Characteristic $2$}\chaptermark{Group Algebra $kA_5$}

\section{Simple $kA_5$-modules}\label{simplekA5mod}

As projective indecomposable modules biject with simple modules, simple modules are important for decomposing $kG$ into a direct sum of projective indecomposable modules. To find all simple $kG$-modules, the next theorem is introduced which tells you that the number of simple $kG$-modules is related to the number of conjugacy classes of $G$. Before starting it, we have a concept called $p$-regular. \index{$p$-regular}For a group $G$, an element $g\in G$ is called \textit{$p$-regular}\index{$p$-regular} if $\ord (g)$ is not divisible by $p$. The following theorems and proofs are based on \cite[Theorem 1.3.2]{Alperin_1986} and \cite{farnsteiner2006simple}. 

\paragraph{The Number of Simple Modules of $kG$}
We first consider the case when $kG$ is semisimple.
\begin{Theorem}\label{semisimplesimpleconjnumber}
Suppose $kG$ is semisimple. Then the number of simple $kG$-modules is the number of ($p$-regular) conjugacy classes of $G$.
\end{Theorem}
\begin{proof}
When an algebra $A$ is semisimple, Artin-Wedderburn's theorem gives a decomposition of $A$ as a direct sum of matrix algebras. If $A$ is a direct sum of subalgebras, then the center $Z(A)$ is the direct sum of each centers. If we write $kG=A_1\oplus \cdots \oplus A_r$ for matrix algebras $A_i$ where $r$ is the number of simple $kG$-modules, then $Z(kG)=Z(A_1)\oplus\cdots\oplus Z(A_r)$. For any size $n$, the centre of a matrix algebra consists of scalar matrices, so if we define $e_i$ as the unit of $A_i$, then $A_i=ke_i$. This tells us that $Z(kG)=ke_1 \oplus \cdots \oplus ke_r$. $\{e_1, \cdots, e_r\}$ is a basis for $Z(kG)$, and $\dim Z(kG)=r$. On the other hand, let $C_1,\cdots,C_s$ be conjugacy classes of $G$ and $\gamma_1,\cdots, \gamma_s$ be sums of each conjugacy classe elements, (i.e., $\gamma_i=\sum_{c\in C_i}c$). By Maschke's Theorem (see \ref{Maschke'sthm}), $|G|$ is not divisible by $p$. This implies all conjugacy classes are $p$-regular, so the number of $p$-regular conjugacy classes of $G$ is the number of conjugacy classes of $G$. To prove $r=s$, we show that $\{\gamma_1, \cdots, \gamma_s\}$ is a basis of $Z(kG)$. For any $g\in G$, $g^{-1}\gamma_ig=\gamma_i$, so $\gamma_i\in Z(kG)$ for all $1\leq i \leq s$. The set is linearly independent since $\gamma_1, \cdots, \gamma_s$ are in different conjugacy classes. Let $a=\sum \alpha_x x$ in $Z(kG)$. As $a=g^{-1}ag$ and $\alpha_{gxg^{-1}}$ is the coefficient of $x$ in $g^{-1}ag$, $\alpha_{gxg^{-1}}=\alpha_x$ for all $x$. This shows that $a$ is a linear combination of $\gamma_1,\cdots, \gamma_s$. We've shown $r=s$, that is, the number of simples is exactly the number of conjugacy classes.
\end{proof}

The theorem is proved by using the center of $kG$, $Z(kG)$. When $kG$ is not necessarily semisimple, the quotient $kG/\Rad kG$ is semisimple by \ref{A/RadA semisimple}. The centre $Z(kG/\Rad kG)$ is a submodule of the quotient. Instead, the concept derived algebra is introduced. For any algebra $A$, the \textit{derived algebra} $[A,A]$ of $A$ is the subspace of $A$ spanned by all elements $ab-ba$ for $a,b\in A$. Explicitly, $[A,A]=\langle ab-ba: a, b\in A\rangle.$ To prove the theorem in the similar way with the quotient $kG/[kG,kG]$, we dualise the concept of the quotient space with the centre: If $A$ is a matrix algebra, then $[A,A]$ is a subspace of codimension $1$ by the following. $\tr(AB-BA)=0$ for all $A,B\in M_n(k)$, so $[M_n(k), M_n(k)]$ consists of all matrices of trace zero. Suppose $kG$ is semisimple. From the decomposition $kG=A_1\oplus\cdots\oplus A_r$ where $A_i$ are matrix subalgebras, we know that $[kG,kG]$ is of codimension $r$. Let $g_1,\cdots, g_s$ be the sums of conjugacy class elements of $kG$. Then $g_1+[kG,kG], \cdots, g_s+[kG,kG]$ are a basis for $kG/[kG,kG]$. By \ref{semisimplesimpleconjnumber} and the duality, $r=s$. Now we prove the following theorem, the general case of $kG$.
 
\begin{Theorem}\label{numofsimpleisequaltonumofpregularconjclasses}
Let $k$ be a field with characteristic $p$.
The number of simple $kG$-modules is precisely the number of $p$-regular conjugacy classes of $G$.
\end{Theorem}

\begin{proof}
Let $T$ be $[kG, kG]$ and $S=T+ \Rad kG$, the sum of two subalgebras $T$ and $\Rad kG$. Then $S/\Rad kG = [kG/\Rad kG, kG/\Rad kG]$ and $S$ is of codimension in $kG$ equal to the number of matrix summands of $kG/\Rad kG$. By \ref{Artinwedderburnthm}, it is the number of simple $kG$-modules. Let $x_1, \cdots, x_s$ be sums of the $p$-regular conjugacy classes of $G$ and $x_{s+1},\cdots, x_r$ be sums of the remaining conjugacy classes. Then $x_1+T, \cdots, x_r+T$ are a basis for $kG/T$ by the following. If $x$ is conjugate to $x_i$, that is, $x=gx_ig^{-1}$ for some $g\in G$, then $x_i-x=x_i-gx_ig^{-1}=g^{-1}\cdot gx_i -gx_i\cdot g^{-1}\in T$. This implies that $x_i+T=x+T$ for all $x\thicksim_c x_i$. As $G$ spans $kG$, the elements $x_1+T, \cdots, x_r+T$ span $kG/T$. For linear independence, define $\varphi_i$ for $1\leq i\leq r$ as the linear function on the vector space $kG$ such that $\varphi_i$ has value 1 on elements of $G$ conjugate to $x_i$ and vanishes on all the other group elements. For any $g,h\in G$, $\varphi_i(gh-hg)=\varphi_i(g(hg)g^{-1}-hg)=0$, so the image of $\varphi_i$ on $T$ is zero. If $\sum \alpha_j(x_j+T)=0$, then $\sum \alpha_jx_j\in T$ and $0=\varphi_i(\sum\alpha_jx_j)=\alpha_i$ for all $i=1,\cdots,r$.

Next we prove that $S=S_0$ where $S_0=\{a\in kG: a^{p^i}\in T \text{ for some integer } i\}.$ For $a\in S_0$, define $\bar{a}$ as the coset of $a$, that is, $a+\rad kG$. Then $\bar{a}^{p^i}$ lies in $[kG/\Rad kG, $ $kG/\Rad kG]$. In the decomposition of $kG/\Rad kG$ into a direct sum of matrix algebras, each component of $\bar{a}$ has $p^i$th power an element of trace zero. Over characteristic $p$, the trace of the $p$th power an element of a matrix is the $p$th power of its trace, since the trace of a matrix is the sum of its eigenvalues. This implies the trace of $\bar{a}$ is zero, so $\bar{a}\in [kG/\Rad kG, kG/\Rad kG]=S/\rad kG$ and $a\in S$. Conversely, $T\subseteq S_0$ and $\Rad kG\subseteq S_0$ as each element of $\Rad kG$ is nilpotent. It remains to show that $S_0$ is closed under addition. Before this, we note that $T$ has the following properties.
\begin{itemize}
\item[(i)]
If $a,b\in kG$, then $(a+b)^p \equiv a^p+b^p \mod T$.

If $a_1,\cdots, a_p\in kG$, then $a_1a_2\cdots a_p+a_2\cdots a_pa_1= a_1a_2\cdots a_p + a_1a_2\cdots a_p +((a_2\cdots a_p)a_1-a_1(a_2\cdots a_p)) \equiv  2a_1\cdots a_p  \mod T.$
So $a_2\cdots a_pa_1 \equiv a_1\cdots a_p\mod T$. It is possible to rearrange a multiple of $p$ elements in modulo $T$. This shows that each term of $(a+b)^p$ except $a^p$ and $b^p$ has coefficient divisible by $p$ in mod $T$, so it is zero and the result follows.
\item[(ii)]
$T^p \subseteq T$.

By (i), $(ab-ba)^p\equiv (ab)^p-(ba)^p=a(ba\cdots b)-(ba\cdots b)a= 0 \mod T$, for all $a, b\in kG$. So $(ab-ba)^p \in T$.
\end{itemize}

Let $a,b\in S_0$. By (ii), there exists a primitive integer $n$ such that both $a^{p^n}, b^{b^n}\in T$.  $(a+b)^{p^n} \equiv a^{p^n}+b^{p^n}=0 \mod T$. This shows $a+b \in S_0$. $S_0$ is closed under addition and we proved $S=S_0$.

Now we claim that $x_1+S,\cdots, x_s+S$ form a basis for $kG/S$. Let $g\in G$ and express $g=ux$ where $u^{p^i}=1$ for some $i \geq 0$ and $p$-regular $x$ such that $ux=xu$. $g$ can be uniquely decomposed in such a way as the cyclic subgroup $\langle g \rangle$ of $G$ and $G$ is the direct product of Sylow subgroups. In modulo $ T$, $((u-1)x)^{p^i}=u^{p^i}x^{p^i}-x^{p^i}=0$, so $(u-1)x=0$ by (i). Then $g=(u-1)x+x \equiv x$, $g^{p^i}\equiv x^{p^i}$, so $(g-x)^{p^i}\equiv 0$ and $g-x\in S_0=S$. Apply $g-x\in S$ to the elements $x_j$ for $j>s$. That is, $x_j\mod S$ is a $p$-regular element which is one of $x_1,\cdots, x_s$ in modulo $T$. This shows the set $\{x_1+S,\cdots, x_s+S\}$ spans $kG/S$.  For linear independence, suppose $\alpha_1x_1+\cdots+\alpha_sx_s\in S$ where $\alpha_j\in k$. By $S=S_0$, $\alpha_1^{p^i}x_1^{p^i}+\cdots+\alpha_x^{p^i}x_s^{p^i}=(\alpha_1x_1+\cdots +\alpha_sx_s)^{p^i}\in T$ for some $i \geq 0$.
$$ \alpha_1^{p^i}x_1^{p^i}+\cdots+\alpha_x^{p^i}x_s^{p^i}=0 \mod T.$$

$x_j^{p^i}=x_j$ modulo $T$ by the following. The order of $G$ can be expressed as the product of a power of $p$ and an integer $m$ with $(p,m)=1$. Choose integers $a$ and $b$ with $ap^i+bm=1$. As $x_j$ are p-regulars, $x_j^m=1$. $({x_j}^{p^i})^a=x_j$ and the integer $a$ is independent of $j$. The equation is then 
$$\alpha_1^{p^i}x_1+\cdots+\alpha_x^{p^i}x_s = 0 \mod T.$$ As $x_1, \cdots, x_s$ are linearly independent in modulo $T$, each $\alpha_j^{p^i}=0 \mod T$, so $\alpha_j=0$. We proved the claim. This shows that $s$ is the codimension of $S$, and the theorem follows as desired. 
\end{proof}

\paragraph{The Number of Simples for $kA_5$}
By \ref{numofsimpleisequaltonumofpregularconjclasses}, the number of simples is the number of $2$-regular conjugacy classes of $A_5$. From $S_5$, we can deduce the conjugacy classes of $A_5$. The conjugacy classes of $S_5$ are classified by their cycle types. Choose representatives of those as $1, (1,2), (1,2)(3,4), (1,2,3), $ $(1,2,3)(4,5), (1,2,3,4)$ and $(1,2,3,4,5)$. The alternating group $A_5$ consists of even permutations in $S_5$, so some representatives of the conjugacy classes are $1,$ $(1,2)(3,4),(1,2,3)$ and $(1,2,3,4,5)$.
\subparagraph{Conjugacy Classes of $A_n$}
In the sign map $\sign: S_n \rightarrow \{\pm 1\}$, the image of $A_n$ is $\{1\}$. $\ker(\sign)=A_n$, so $[S_n:A_n]=2$. Fix $g\in G$. Consider the restriction map $\sign|_{C_{S_n}(g)}: C_{S_n}(g) \rightarrow \{\pm 1\}$.
The image of $\sign|_{C_{S_n}(g)}$ is either $\{\pm 1\}$ or $\{1\}$.
It maps $C_{A_n}(g)$ onto $\{1\}$, so $\ker(\sign|_{C_{S_n}(g)})=C_{A_n}(g)$. It follows that $[C_{S_n}(g): C_{A_n}(g)]$ is either $1$ or $2$.
\begin{itemize}
\item[(1)] Assume $|C_{S_n}(g)|=2|C_{A_n}(g)|$. Then, by \ref{ordofconjugacyclassesofGisordofG/centerofg},
$$\frac{|g^{S_n}|}{|g^{A_n}|}=\frac{|S_n|}{|A_n|}\cdot\frac{|C_{A_n}(g)|}{|C_{S_n}(g)|}=1.$$
i.e., $|g^{S_n}|=|g^{A_n}|$. As $g^{A_n}\leq g^{S_n}$, $g^{S_n}=g^{A_n}$.

\item[(2)] When $|C_{S_n}(g)|=|C_{A_n}(g)|$, by \ref{ordofconjugacyclassesofGisordofG/centerofg},
$$\frac{|g^{S_n}|}{|g^{A_n}|}=\frac{|S_n|}{|A_n|}\cdot\frac{|C_{A_n}(g)|}{|C_{S_n}(g)|}=2.$$
i.e., $|g^{S_n}|=2|g^{A_n}|$. $g^{S_n}=g^{A_n}\dot{\cup}(g^x)^{A_n}$ for any $x \in S_n\backslash A_n$.
\end{itemize}
If $g^{(1,2)}\neq g$, then $g$ is in case $(2)$, so the conjugacy class $g^{S_n}$ is divided into two conjugacy classes $g^{A_n}$ and $(g^x)^{A_n}$ with equal size. Otherwise, $g^{A_n}=g^{S_n}$.

\subparagraph{2-Regular Conjugacy Classes of $A_5$}
$(1,2)(3,4)$ commutes with $(1,2)$, so $(1,2)$ $(3,4)^{A_5}$ is ${(1,2)(3,4)}^{S_5}$. $(1,2,3)^{(1,2)}$ $=(1,2,3)$, so $(1,2,3)^{A_5}=(1,2,3)^{S_5}$.
However, $(1,2,3,4,5)^{(1,2)}=(1,3,5,4,2)$, so $(1,2,3,4,5)^{A_5}$ and $(1,3,4,5,2)^{A_5}$ are different conjugacy classes. $A_5$ has $5$ conjugacy classes with representatives $1$, $(1,2)(3,4)$, $(1,2,3)$, $(1,2,3,4,5)$ and $(1,3,4,5,2)$. $(1,2)(3,4)$ is not $2$-regular, so $A_5$ has four $2$-regular conjugacy classes. By \ref{numofsimpleisequaltonumofpregularconjclasses}, $kA_5$ has four simple modules.

\paragraph{Simple $kA_5$-modules}\label{S2}\label{S3}\label{S4}
Let $S_1=k_{A_5}$, $S_2$, $S_3$ and $S_4$ be the simple $kA_5$-modules. For $S_2$, identify points in $\{1,\cdots, 5\}$ with $1$-dimensional subspace of ${\mathbb{F}_4}^2$ by following: $1$ with $\left\{ a\big[\begin{smallmatrix} 1\\1   \end{smallmatrix}\big]: a\in\mathbb{F}_4\right\}$, $2$ with $\left\{ a \big[\begin{smallmatrix} 1\\ \omega \end{smallmatrix}\big] : a\in\mathbb{F}_4\right\}$, $3$ with $\left\{ a \,\big[\begin{smallmatrix} 1\\ {\omega}^2   \end{smallmatrix}\big] : a\in\mathbb{F}_4\right\}$, $4$ with $\big\{ a \big[\begin{smallmatrix} 0\\1   \end{smallmatrix}\big]: a\in\mathbb{F}_4\big\}$ and $5$ with $\left\{ a \big[\begin{smallmatrix} 1\\0   \end{smallmatrix}\big]: a\in\mathbb{F}_4\right\}$. Then $A_5 \cong SL_2(\mathbb{F}_4)$.  

By extending from $\mathbb{F}_4$ to $\mathbb{F}_4$-action on $S_2$, the 2-dimensional simple module $S_2$ is defined as the irreducible representation $S_2: A_5 \rightarrow SL_2(\mathbb{F}_4)\subseteq GL_2(k)$ where $(1,2)(3,4) \mapsto \big[\begin{smallmatrix} 0 & 1 \\  1 & 1 \end{smallmatrix}\big]$, $(1,2,3) \mapsto \big[\begin{smallmatrix}  \omega &0 \\ 0 & {\omega}^2 \end{smallmatrix}\big]$, and $(1,3,5) \mapsto \big[\begin{smallmatrix}  1 &0 \\ \omega & 1 \end{smallmatrix}\big]$. $S_2$ is simple because if there exists a non-trivial submodule, then it is 1-dimensional so is an eigenspace of some eigenvalue $\lambda$. But there is no common eigenvalue for all, so this leads to a contradiction. The other 2-dimensional simple module $S_3$ is obtained by substituting $\omega^2$ for $\omega$ in $S_2$, that is, $S_3: A_5 \rightarrow SL_2(\mathbb{F}_4)$ where $(1,2)(3,4) \mapsto \big[\begin{smallmatrix} 1 &\omega^2\\ 0 & 1 \end{smallmatrix}\big]$
, $(1,2,3) \mapsto \big[\begin{smallmatrix}  {\omega}^2 &0 \\ 0 & \omega \end{smallmatrix}\big]$, $(1,3,5) \mapsto \big[\begin{smallmatrix}  0 &1 \\ 1 & 1 \end{smallmatrix}\big]$. By the structural symmetry with $S_2$, $S_3$ is simple by the same reason as $S_2$.

Let $W$ be a permutation module from action on $\{1,2,3,4,5\}$. To avoid any confusion, let ${S_5}'$ be the symmetric group of the set. Then $W$ is a representation $W: {S_5}' \rightarrow GL_5(k)$ where the image of each element in $S_5'$ has every entry either 0 or 1. By considering eigenspaces, a basis of $W$ is $\{(1,0,0,0,0)$, $(0,1,0,0,0)$, $(0,0,1,0,0)$, $(0,0,0,1,0)$, $(0,0,0,0,1)\}$, and this implies $W$ is free. The 4-dimensional simple module $S_4$ is a submodule of $W$ with a basis $\{(1,1,0,0,0)$, $(0,1,1,0,0)$, $(0,0,1,1,0)$, $(0,0,0,1,1)\}$. A basis for $k_{A_5}$ is $\{(1,1,1,1,1)\}$, so $W$ is decomposed as $W= k_{A_5}\oplus S_4$. This shows that $S_4$ is a direct summand of the free module $W$, so it is projective. By \ref{Simpleindecomp}, it is indecomposable. $S_4$ is a projective indecomposable $kA_5$-module, so we can set a projective indecomposable module $P_4=S_4$ without loss of generality. We've found all simple $kA_5$-modules.

\section{Indecomposable $kA_5$-modules}

Like we did in section \ref{IndecompM_iandPM_iIM_i}, in order to find projective indecomposable modules of $A_5$, we want to find some indecomposable modules first. As we know the radical series of all simple $kA_4$-modules $T_j$ and simple $kA_5$-modules $S_i$, the radical series of the induced module ${T_j}^{A_5}$ on $T_j$ to $A_5$ is good to be found. We use Frobenius reciprocity on $T_j$ and $S_i$ to find the radical series structure of ${T_j}^{A_5}$. It is deduced by investigating $\Hom_{kA_4}(S_i, T_j^{A_5})$ and $\Hom_{kA_5}(T_j^{A_5}, S_i)$. In the context, $A_4$ is considered as a subgroup of $A_5$, as it is isomorphic to a subgroup of $A_5$ consisting of permutations stabilising one fixed point in $\{1,2,3,4,5\}$. Overall discussion and results in this section refers to \cite{C.W.Eaton}.

\begin{Theorem}[Frobenius Reciprocity for Modules]
Let $H$ be a subgroup of a group $G$. Then, for any $kG$-module $V$ and $kH$-module $W$,
$$\begin{matrix} \Hom_{kG}(W^G, V) \cong \Hom_{kH}(W, (V)_H) & \text{and} & \Hom_{kG}(V, W^G) \cong \Hom_{kH}((V)_H, W)\end{matrix}$$
as vector spaces.
\end{Theorem}
\begin{proof}
The mutually inverse isomorphisms of $ \Hom_{kG}(W^G, V) \cong \Hom_{kH}(W, (V)_H)$ are
\begin{align*}
&\quad\quad\quad\quad\quad\quad\quad\; f \mapsto (v\mapsto f(1\otimes v)) \\
\quad \text{and} \quad&\quad (b\otimes v \mapsto bg(v)) \leftarrow g.
\end{align*}
\end{proof}

For every simple $kA_5$-modules $S_i$ and $kA_4$-modules $T_j$,
\begin{align}\label{SiTj}
\dim_k(\Hom_{kA_5}(S_i, T_j^{A_5}))=&\dim_k(\Hom_{kA_4}((S_i)_{A_4}, T_j))\\
\label{TjSi}\dim_k(\Hom_{kA_5}(T_j^{A_5}, S_i))=&\dim_k(\Hom_{kA_4}(T_j, (S_i)_{A_4})).
\end{align}
The restrictions on $S_i$ to $A_4$ are
\begin{align}\label{S1A4}
(S_1)_{A_4}=&(S_4)_{A_4}=T_1=k_{A_4},\\
\label{S1A4'}(S_2)_{A_4}=&\begin{smallmatrix} T_2 \\ T_3 \end{smallmatrix}\\
(S_3)_{A_4}=&\begin{smallmatrix} T_3 \\ T_2 \end{smallmatrix}
 \end{align}
 where ${k_{A_4}}=(W)_{A_4}=(S_1\oplus S_4)_{A_4}$, $(S_2)_{A_4}$ is $(1,2,3) \mapsto \big[\begin{smallmatrix}  \omega &0 \\ 0 & {\omega}^2 \end{smallmatrix}\big]$ and $(S_3)_{A_4}$ is $(1,2,3) \mapsto \big[\begin{smallmatrix} {\omega}^2 &0 \\ 0 & \omega \end{smallmatrix}\big]$ in \ref{S3}. $(S_2)_{A_4}$ is uniserial and has the unique quotient isomorphic to $T_2$ and the unique submodule isomorphic to $T_3$. Similarly, $(S_3)_{A_4}$ is uniserial and has the unique quotient $T_3$ and the unique submodule $T_2$.

\subparagraph{Induced Simple $kA_4$-Modules, $T_i^{A_5}$}
The following results refer to \cite{C.W.Eaton}. We want to find $T_i^{A_5}$. $T_1^{A_5}$ is decomposable by $T_1^{A_5}={k_{A_4}}^{A_5}=W=k_{A_5} \oplus S_4=S_1\oplus S_4$. To find submodules of $T_2^{A_5}$, we consider $\Hom_{kA_5}(S_i,T_2^{A_5})$.
By the relation in \ref{S1A4}, $\Hom_{kA_4}((S_1)_{A_4}, T_2)=\Hom_{kA_4}((S_4)_{A_4}, T_2)=\Hom_{kA_4}(T_1, T_2)$, so all zero by Schur's Lemma. A nonzero map $\begin{smallmatrix} T_2 \\ T_3 \end{smallmatrix} \rightarrow T_2$ exists only when $T_3$ is mapped into zero, so $\dim_k\Hom_{kA_4}((S_2)_{A_4}, T_2)=1$. $\Hom_{kA_4}((S_3)_{A_4}, T_2)=0$ as no homomorphism $\begin{smallmatrix} T_3 \\ T_2 \end{smallmatrix} \rightarrow T_2$ exists that maps $T_3$ into $T_2$. By the equation \ref{SiTj},
\begin{displaymath}
    \dim_k\Hom_{kA_5}(S_i, T_2^{A_5})=
     \left\{
    \begin{array}{lr}
      1  \quad\quad i=2\\
      0  \quad\quad i=1,3,4.
    \end{array}
  \right.
\end{displaymath}
This implies that $S_2$ is the unique simple submodule of $T_2^{A_5}$, i.e, $\soc T_2^{A_5}=S_2$. To find quotients of $T_2^{A_5}$, we consider $\Hom_{kA_5}(T_2^{A_5},S_i)$. By \ref{S1A4}, $\Hom_{kA_4}(T_2, (S_1)_{A_4})=\Hom_{kA_4}(T_2, (S_4)_{A_4})=\Hom_{kA_4}(T_2, T_1)=0$ by Schur's Lemma. $\Hom_{kA_4}(T_2, (S_2)_{A_4})=0$ as $T_2$ has no submodule mapped onto $T_3$ for any nonzero homomorphism $T_2\rightarrow \begin{smallmatrix} T_2 \\ T_3 \end{smallmatrix}$. A nonzero morphism $T_2 \rightarrow \begin{smallmatrix} T_3 \\ T_2 \end{smallmatrix}$ exists only when $T_2$ is mapped onto $T_2$, so $\dim_k\Hom_{kA_4}(T_2, (S_3)_{A_4}$ $)=1$. By the equation \ref{TjSi},
\begin{displaymath}
    \dim_k\Hom_{kA_5}(T_2^{A_5}, S_i)=
     \left\{
    \begin{array}{lr}
      1  \quad\quad i=3\\
      0  \quad\quad i=1,2,4.
    \end{array}
  \right.
\end{displaymath}
This implies that $S_3$ is the unique simple quotient of $T_2^{A_5}$, i.e, $T_2^{A_5}/\rad T_2^{A_5}=S_3$.
\begin{equation}\label{T2A51stlast} T_2^{A_5}/\Rad (T_2^{A_5})\cong S_3 \quad\text{and}\quad  \soc (T_2^{A_5}) \cong S_2\end{equation}
For the remaining layers, consider the dimension of $T_i^{A_5}$.
\label{TiA5dim}
By equation \ref{inddim} and $[A_5:A_4]=|A_5|/|A_4|=5$, $\dim T_i^{A_5}=[A_5:A_4]\cdot\dim T_i=5$. This implies $S_1$ is the remaining layer. Hence, the radical series of $T_2^{A_5}$ is \begin{equation}T_2^{A_5}:\quad\begin{matrix} S_3 \\ S_1 \\ S_2 \end{matrix}.\end{equation}

We do similarly to find $T_3^{A_5}$. In $\Hom_{kA_4}((S_i)_{A_4},T_3)$, $\Hom_{kA_4}((S_1)_{A_4}, T_3)=\Hom_{kA_4}$ $((S_4)_{A_4}, T_3)=\Hom_{kA_4}(T_1, T_3)=0$ by Schur's Lemma. A nonzero homomorphism $\begin{smallmatrix} T_2 \\ T_3 \end{smallmatrix}\rightarrow T_3$ does not exist as $T_2$ cannot be mapped onto $T_3$, so $\Hom_{kA_4}((S_2)_{A_4}, T_3)=0$. A nonzero morphism $\begin{smallmatrix} T_3 \\ T_2 \end{smallmatrix} \rightarrow T_3$ exists only when $T_3$ is mapped onto $T_3$ and $T_2$ into zero, so $\dim_k\Hom_{kA_4}((S_3)_{A_4}, T_3)=1$. By the equation \ref{SiTj},
\begin{displaymath}
    \dim_k(\Hom_{kA_5}(S_i, T_3^{A_5}))=
     \left\{
    \begin{array}{lr}
      1  \quad\quad i=3\\
      0  \quad\quad i=1,2,4.
    \end{array}
  \right.
\end{displaymath}
and this implies that $\soc T_3^{A_5}=S_3$. In $\Hom_{kA_5}(T_3^{A_5}, S_i)$, $\Hom_{kA_4}(T_3, (S_1)_{A_4})=\Hom_{kA_4}$ $(T_3, (S_4)_{A_4})=\Hom_{kA_4}(T_3, T_1)=0$ by Schur's Lemma. A nonzero morphism $T_3 \rightarrow \begin{smallmatrix} T_2 \\ T_3 \end{smallmatrix}$ exists only when $T_3$ is mapped onto $T_3$, so $\dim_k\Hom_{kA_4}(T_3, (S_2)_{A_4})=1$. A nonzero map $T_3\rightarrow \begin{smallmatrix} T_3 \\ T_2 \end{smallmatrix}$ does not exist as no submodule exists mapped onto $T_2$, so $\Hom_{kA_4}(T_3, (S_3)_{A_4})=0$. By the equation \ref{TjSi},
\begin{displaymath}
    \dim_k(\Hom_{kA_5}(T_3^{A_5}, S_i))=
     \left\{
    \begin{array}{lr}
      1  \quad\quad i=2\\
      0  \quad\quad i=1,2,4.
    \end{array}
  \right.
\end{displaymath}
and thus, $T_3^{A_5}/\rad T_3^{A_5}=S_2$.
\begin{equation}\label{T3A51stlast} T_3^{A_5}/\rad (T_3^{A_5})\cong S_2  \quad\text{and}\quad  \soc(T_3^{A_5})\cong S_3.\end{equation}
As $\dim T_3^{A_5}=5$, the dimension of the remaining layers is 1, so $T_3^{A_5}$ has the only one remaining radical layer, that is, $S_1$. Hence, its radical series is \begin{equation}T_3^{A_5}:\quad\begin{matrix} S_2 \\ S_1 \\ S_3 \end{matrix}.\end{equation}
For remark, $T_2^{A_5}$ and $T_3^{A_5}$ are uniserial so are indecomposable.

\section{Projective Indecomposable $kA_5$-modules}

By \ref{regmodAdecompassumofprojcovers} and the result of subsection \ref{simplekA5mod}, $kA_5$ has four projective indecomposable modules. Denote  $P_i$ for $i=1,\cdots,4$ as the projective indecomposable modules of $A_5$. The radical series of $P_i$ are
\begin{equation}\label{Pipossiblestructure}\begin{matrix} P_i=\quad \begin{matrix} S_i \\ * \\ S_i \end{matrix}\quad \quad\text{for } i=1,2,3 & \quad\text{and} \quad& P_4=S_4\end{matrix}.\end{equation}

\subparagraph{Projective Covers and Injective Envelopes of $T_2^{A_5}$ and $T_3^{A_5}$}
Recall that $P_i=P_{S_i}=I_{S_i} \quad i=1,2,3, \quad \text{and} \quad  P_4=P_{S_4}=S_4.$ Each of the projective cover $P_{T_i^{A_5}}$ and the injective envelope $I_{T_i^{A_5}}$ of $T_i^{A_5}$ is one of $P_1$, $P_2$, $P_3$ and $P_4$ by the following. From the relation \ref{T2A51stlast}, $P_{T_2^{A_5}}=P_3$ as, among $P_i$,
\begin{equation*}P_3\rightarrow S_3(\cong T_2^{A_5}/\Rad T_2^{A_5})
\end{equation*}
 is the only essential epimorphism that maps $P_{T_2^{A_5}}$ onto $ T_2^{A_5}/\Rad T_2^{A_5}$. $I_{T_2^{A_5}}=P_2$ as, among $P_i$,
 \begin{equation*}(\soc (T_2^{A_5})\cong) S_2 \rightarrow P_2\end{equation*}
  is the only essential monomorphism that maps $\soc T_2^{A_5}$ into $I_{T_2^{A_5}}$. Similarly, from the relation \ref{T3A51stlast}, $P_{T_3^{A_5}}=P_2$ as, among $P_i$, only $P_2$ maps $P_{T_3^{A_5}}$ onto $S_2\cong T_3^{A_5}/\rad T_3^{A_5}$. $I_{T_3^{A_5}}=P_3$ as, among $P_i$, only $I_3=P_3$ maps $S_3\cong \soc T_3^{A_5}$ into $I_{T_3^{A_5}}$.
\begin{align}\label{P2projinj}P_2 = P_{T_3^{A_5}} =I_{T_2^{A_5}}\\ \label{P3projinj}P_3 = P_{T_2^{A_5}} =I_{T_3^{A_5}}.\end{align}

The following gives a way to find the dimensions of $P_i$. By \ref{dimAbydecompproj}, $60 = \dim P_1 + 2\dim P_2 + 2 \dim P_3 + 4\cdot 4$, so $\dim P_1 + 2\dim P_2 + 2\dim P_3 = 44.$ $\dim {P_i}$ is divided by 4 by \ref{largpowerofpdivdesdimofprojmod}. By the structural symmetry of $S_2$ and $S_3$, $\dim P_2=\dim P_3$. We do simple calculation and obtain that the dimensions of $P_i$ are either
\begin{align*}& \dim P_1=12, \;\dim P_2=\dim P_3=8 &\text{or  }\quad& \dim P_1=28,\; \dim P_2=\dim P_3=4.\end{align*}
 By \ref{Pipossiblestructure}, the radical series of $P_2$ and $P_3$ has at least one radical layer other than the head and the socle, so their dimensions are greater than $4$. Hence, $$\begin{matrix} \dim P_1 = 12,& \dim P_2 = 8, &\text{and} &\dim P_3=8\end{matrix}.$$

\subparagraph{Radical Series of PIMs}
The results of $P_i$ are from \cite{C.W.Eaton}. By the relation \ref{P2projinj}, $P_2$ has a quotient $T_3^{A_5}: \begin{smallmatrix} S_2 \\ S_1 \\ S_3 \end{smallmatrix}$ and a submodule $T_2^{A_5}: \begin{smallmatrix} S_3 \\ S_1 \\ S_2 \end{smallmatrix}$. The only possible radical series of $P_2$ with the least dimension is the uniserial module $\begin{smallmatrix} S_2 \\ S_1 \\ S_3 \\ S_1 \\ S_2 \end{smallmatrix}$ with dimension $8$, equal to $\dim{P_2}$. This shows $P_2$ is precisely the uniserial module.
We do in the similar way to find the radical series of $P_3$. By the relation \ref{P3projinj}, $P_3$ has a quotient $T_2^{A_5}: \begin{smallmatrix} S_3 \\ S_1 \\ S_2 \end{smallmatrix}$ and a submodule $T_3^{A_5}: \begin{smallmatrix} S_2 \\ S_1 \\ S_3 \end{smallmatrix}$. Then $\begin{smallmatrix} S_3 \\ S_1 \\ S_2 \\ S_1 \\ S_3 \end{smallmatrix}$ is the only possible radical series of $P_3$ having the least dimension and the dimension is equal to $P_3$, 8. This shows $P_3$ is the uniserial module.
\begin{figure}[h]\begin{align*} P_2 =& \quad \begin{matrix} S_2 \\ S_1 \\ S_3 \\ S_1 \\ S_2 \end{matrix} &
P_3=& \quad \begin{matrix} S_3 \\ S_1 \\ S_2 \\ S_1 \\ S_3 \end{matrix}\end{align*}\end{figure}

\subparagraph{\boldmath$P_1$}
We want to find the radical series of $P_1$, specifically, each radical layer of $P_1$. In order to do this, we consider $\Hom_{kA_5}(P_1, P_i)$ and $\Hom_{kA_5}(P_i, P_1)$ for $i=2, 3$.
In $\Hom_{kA_5}(P_1, P_2)$,
there is a homomorphism $\varphi_1: P_1\rightarrow P_2$ where the first radical layer $P_1/\Rad P_1$ isomorphic to $S_1$ is mapped onto the second radical layer $\Rad P_2/{\Rad^2 P_2}$ isomorphic to $S_1$ and the last layer $\Soc P_1$ is mapped into $0$. $P_1$ has a quotient $P_1/\ker(\varphi_1)\cong \im(\varphi_1)$, which is $\begin{smallmatrix} S_1\\S_3\\S_1\\S_2 \end{smallmatrix}$. There exists the other homomorphism $\varphi_2$ where the first layer $P_1/\Rad P_1\cong S_1$ is mapped onto the fourth layer $\Rad^3 P_2/\Rad^4 P_2\cong S_1$ and $\Soc P_1\cong S_1$ is mapped into zero. $P_1$ has a quotient $P_1/\ker \varphi_2\cong\im(\varphi_2)  =\begin{smallmatrix}
S_1\\S_2
\end{smallmatrix}$. 
Up to scalar, $\varphi_1$ and $\varphi_2$ are the only homomorphisms in $\Hom_{kA_5}(P_1,P_2)$.
\begin{figure}[h]
\begin{align*}
\begin{matrix}
\varphi_1: &P_1     && \longrightarrow  && P_2 \\ 
&&&&&\\
&S_1     &&                    && S_2\\
&            &&   \searrow   &&S_1\\
& *          &&                    &&S_3 \\
 &           &&                     &&S_1\\
& S_1    & &                    &&S_2\\
 &            & &  \searrow   &&0     
 \end{matrix} \quad\quad&\quad\quad \begin{matrix}
\varphi_2: &P_1 && \longrightarrow && P_2 \\ 
&&&&&\\
&S_1 &&             && S_2 \\
 &&&     &&S_1\\
& * && \searrow &&S_3 \\
 &&   &&&S_1\\
& S_1&  &&&S_2\\
& &   &\searrow  & &0  
 \end{matrix}.
 \end{align*}
 \end{figure}

To find submodules of $P_1$, we consider $\Hom_{kA_5}(P_2, P_1)$. There is a homomorphism $\sigma_1:P_2\rightarrow P_1$ that maps the fourth radical layer $\rad^3 P_2/\rad^4 P_2\cong S_1$ onto $\soc P_1\cong S_1$. The kernel of $\sigma_1$ is $\soc P_2\cong S_2$. $P_1$ has the submodule $\im \sigma_1\cong P_2/\ker \sigma_1=\begin{smallmatrix}S_2\\S_1\\S_3\\S_1\end{smallmatrix}$. There exists the other homomorphism $\sigma_2: P_2\rightarrow P_1$ that maps the second radical layer $\rad P_2/\rad^2 P_2\cong S_1$ onto the last layer $\soc P_1\cong S_1$. The kernel of $\sigma_2$ is $\begin{smallmatrix}S_3\\S_1\\S_2 \end{smallmatrix}$. $P_1$ has the submodule $\im \sigma_2\cong P_2/\ker \sigma_2=\begin{smallmatrix}S_2\\S_1\end{smallmatrix}$.

\begin{figure}[h]\centering
\begin{align*}
\begin{matrix}
\sigma_1: &P_2     && \longrightarrow  && P_1 \\ 
&&&&&\\
&S_2     &&                    && S_1\\
&S_1            &&   \searrow   &&\\
& S_3          &&                    &&* \\
 & S_1          &&                     &&\\
& S_2    & &        \searrow              &&S_1\\
 &            & &   &&0     
 \end{matrix} \quad\quad&\quad\quad
 \begin{matrix}
\sigma_2: &P_2     && \longrightarrow  && P_1 \\ 
&&&&&\\
&S_2     &&                    && S_1\\
&S_1            &&      &&\\
& S_3          &&          \searrow           &&* \\
 & S_1          &&                     &&\\
& S_2    & &                     &&S_1\\
 &            & &   &&0     
 \end{matrix} 
 \end{align*}
 \end{figure}

In $P_1$, there are quotients \;
$\begin{matrix}
S_1\\S_3\\S_1\\S_2
\end{matrix}$ \;and\; $\begin{matrix}
S_1\\S_2
\end{matrix}$ and submodules  \;\;$\begin{matrix}S_2\\S_1\\S_3\\S_1\end{matrix}$ \;and\; 
$\begin{matrix}S_2\\S_1\end{matrix}$\;.

We do the same on $\Hom_{kA_5}(P_1, P_3)$ and $\Hom_{kA_5}(P_3, P_1)$. In $\Hom_{kA_5}(P_1, P_3)$, one nonzero homomorphism is $\eta_1: P_1\rightarrow P_3$ where the first radical layer of $P_1$ maps onto the second radical layer of $P_3$ and the last layer into $0$. $P_1$ has the quotient $P_1/\ker \eta_1\cong\im \eta_1=\begin{smallmatrix}S_1\\S_2\\S_1\\S_3\end{smallmatrix}$. The other homomorphism is $\eta_2: P_1\rightarrow P_3$ where the first radical layer maps onto the fourth radical layer and the last layer into $0$. $P_1$ has the quotient $P_1/\ker \eta_2\cong \im \eta_2\cong\begin{smallmatrix}S_1\\S_3 \end{smallmatrix}$. Up to scalar, $\eta_1$ and $\eta_2$ are the only homomorphisms in $\Hom_{kA_5}(P_1, P_3)$.

\begin{figure}[h]
\begin{align*}
\begin{matrix}
\eta_1: &P_1 && \longrightarrow && P_3  \\ 
&&&&&\\
&S_1 &&             && S_3 \\
 &      &&   \searrow  &&S_1\\
& * &&                   &&S_2 \\
& &&        & &S_1\\
& S_1 &&               &&S_3\\
& &&   \searrow  & &0 
 \end{matrix} \quad\quad&\quad\quad
\begin{matrix}
\eta_2: &P_1 && \longrightarrow && P_3 \\ 
&&&&&\\
&S_1 &&             && S_3\\
 &&&     &&S_1\\
& * && \searrow &&S_2\\
 &&&&&S_1\\
& S_1&&&&S_3\\
& & &  \searrow  & &0
 \end{matrix}.
 \end{align*}
 \end{figure}
In $\Hom_{kG}(P_3, P_1)$, there is a homomorphism $\tau_1:P_3\rightarrow P_1$ that maps the second radical layer onto some middle radical layer of $P_1$ and the fourth radical layer onto the last layer $\soc P_1$. The kernel of $\tau_1$ is $\soc P_3\cong S_3$.  $P_1$ has the submodule $\im \tau_1\cong P_3/\ker \tau_1=\begin{smallmatrix}S_3\\S_1\\S_2\\S_1\end{smallmatrix}$. There exists the other homomorphism $\tau_2: P_3\rightarrow P_1$ that maps the second radical layer onto the last layer $\soc P_1$. $\ker(\tau_2)=\begin{smallmatrix}S_2\\S_1\\S_3 \end{smallmatrix}$. $P_1$ has the submodule $\im \tau_2\cong P_3/\ker \tau_2=\begin{smallmatrix}S_3\\S_1\end{smallmatrix}$.

\begin{figure}[h]
\begin{align*}
\begin{matrix}
\tau_1: &P_3     && \longrightarrow  && P_1 \\ 
&&&&&\\
&S_3     &&                    && S_1\\
&S_1            &&     &&\\
& S_2          &&     \searrow                &&* \\
 & S_1          &&                     &&\\
& S_3    & &        \searrow              &&S_1\\
 &            & &   &&0     
 \end{matrix} \quad\quad & \quad\quad
\begin{matrix}
\tau_2: &P_3     && \longrightarrow  && P_1 \\ 
&&&&&\\
&S_3     &&                    && S_1\\
&S_1            &&    &&\\
& S_2          &&          \searrow           &&* \\
 & S_1          &&                     &&\\
& S_3    & &                     &&S_1\\
 &            & &   &&0     
 \end{matrix} \end{align*} \end{figure}
 
 We obtained quotients of $P_1$, \;\;
$\begin{matrix}
S_1\\S_2\\S_1\\S_3
\end{matrix}$ \;\; and\;\; 
$\begin{matrix}
S_1\\S_3.
\end{matrix}$\;,
and submodules of $P_1$, \;\;
$\begin{matrix}S_3\\S_1\\S_2\\S_1\end{matrix}$
\;\;and\;\;
$\begin{matrix}S_3\\S_1\end{matrix}$\;\;. Hence, we have a possible $P_1$ radical series structure
\begin{figure}[h]
\centering
$\begin{matrix}
&S_1\\
  S_3&&S_2\\
  S_1  & & S_1\\
S_2&&S_3\\
& S_1
 \end{matrix}$\;.
 \end{figure}
 
By the diagram \ref{Q1projindradser}, a short exact sequence $0\rightarrow \begin{smallmatrix} T_3\\T_1\end{smallmatrix}\rightarrow Q_1 \rightarrow \begin{smallmatrix} T_1\\T_2\end{smallmatrix} \rightarrow 0$ satisfies $Q_1=\begin{smallmatrix} T_3\\T_1\end{smallmatrix} \oplus \begin{smallmatrix} T_1\\T_2\end{smallmatrix}$. Recall that an induced projective module is projective. By projectivity of ${Q_1}^{A_5}$, the short exact sequence
\begin{equation*}\label{shortexactseqQ1G}0\rightarrow {\begin{smallmatrix} T_3\\T_1\end{smallmatrix}}^{A_5} \rightarrow {Q_1}^{A_5} \rightarrow {\begin{smallmatrix} T_1\\T_2\end{smallmatrix}}^{A_5} \rightarrow 0\end{equation*}
satisfies
\begin{equation*}
{Q_1}^{A_5}={\begin{smallmatrix} T_3\\T_1\end{smallmatrix}}^{A_5} \oplus{\begin{smallmatrix} T_1\\T_2\end{smallmatrix}}^{A_5}.
\end{equation*}
This implies that ${Q_1}^{A_5}$ has a submodule ${\begin{smallmatrix} T_3\\T_1\end{smallmatrix}}^{A_5}=\begin{smallmatrix} \begin{smallmatrix} S_2 \\ S_1 \\ S_3 \end{smallmatrix}\\S_1\oplus S_4\end{smallmatrix}$ and its quotient ${Q_1}^{A_5}/{\begin{smallmatrix} T_3\\T_1\end{smallmatrix}}^{A_5}= {\begin{smallmatrix} T_1\\T_2\end{smallmatrix}}^{A_5}=\begin{smallmatrix} \begin{smallmatrix} S_1\oplus S_4 \\ S_3 \\ S_1 \\ S_2 \end{smallmatrix}\end{smallmatrix}$.
 Hence, the radical series of ${Q_1}^{A_5}$ is
\begin{figure}[H]\centering
\begin{align*}{Q_1}^{A_5}=& \quad
\left(\begin{matrix} &T_1& \\T_2 & \oplus & T_3 \\  &T_1& \;\;. \end{matrix}\right)^{A_5}\quad=
\quad \begin{matrix}
&S_1\oplus S_4\\
  S_3&&S_2\\
  S_1  & \oplus &S_1 \\
S_2&&S_3\\
& S_1\oplus S_4
 \end{matrix} \quad = \quad \begin{matrix}
&S_1&\\
  S_3&&S_2&\\
  S_1  & \oplus &S_1 &\oplus\;\; {S_4}^2 \;.\\
S_2&&S_3&\\
& S_1&
 \end{matrix} 
 \end{align*}\end{figure}
Thus, we proved that $P_1$ is 
\begin{figure}[H]
\begin{align*}     P_1=&\quad\begin{matrix}
&S_1& \\
S_3 & &S_2\\
S_1&\oplus&S_1\\
S_2&&S_3\\
&S_1& .
 \end{matrix} \end{align*}\end{figure} Here, $\rad P_1/\soc P_1$ is decomposable. Specifically, $\rad P_1/\soc P_1={T_2}^{A_5}\oplus{T_3}^{A_5}$.
 The projective indecomposable $kA_5$-modules are
 \begin{figure}[h]
 \centering
 \begin{align}\label{allPi}
  P_1= \begin{matrix}
&S_1& \\
S_3 & &S_2\\
S_1&\oplus&S_1\\
S_2&&S_3\\
&S_1&
 \end{matrix}, \quad&\quad\;\; P_2=\quad \begin{matrix} S_2 \\ S_1 \\ S_3 \\ S_1 \\ S_2 \end{matrix}, & P_3=\quad \begin{matrix} S_3 \\ S_1 \\ S_2 \\ S_1 \\ S_3 \end{matrix}, \quad&\quad P_4= S_4.\end{align}
 \caption{The projective indecomposable modules of $A_5$}
\end{figure}

\subparagraph{Decomposition of $kA_5$} The decomposition of $kA_5$ as the direct sum of projective indecomposable modules is 
\begin{equation*}
kA_5 \cong {P_1}^1 \oplus {P_2}^2 \oplus {P_3}^2 \oplus {P_4}^4
\end{equation*}
\begin{figure}[h]
\begin{align*}
\text{where} \quad&\;\; P_1= \begin{matrix}
&S_1& \\
S_3 & &S_2\\
S_1&\oplus&S_1\\
S_2&&S_3\\
&S_1&
 \end{matrix}, & P_2=\quad \begin{matrix} S_2 \\ S_1 \\ S_3 \\ S_1 \\ S_2 \end{matrix}, \quad& \quad P_3=\quad \begin{matrix} S_3 \\ S_1 \\ S_2 \\ S_1 \\ S_3 \end{matrix}, & \text{and} \quad& P_4= S_4.
 \end{align*}
 \end{figure}

\subparagraph{The Cartan Matrix of $kA_5$}
By calculating the multiplicity of $S_i$ in $P_j$ for $1\leq i,j \leq 4$ in the diagram \ref{allPi}, the Cartan matrix $(C_{S_iS_j})$ of $kA_5$ is \begin{figure}[h]\label{cartanmatrixA5}
\centering
$\left[\begin{matrix} 4 & 2 & 2 & 0 \\ 2 & 2 & 1 & 0 \\ 2 & 1 & 2 & 0\\ 0 & 0 & 0 & 1 \end{matrix}\right].$
\caption{The Cartan matrix of $kA_5$}
\end{figure}

\subparagraph{Decompositions of $kA_5$ into Blocks} By the above Cartan matrix \ref{cartanmatrixA5}, the multiplicity of $S_i$ in $P_j$ is nonzero when $1\leq i,j \leq 3$. By \ref{blocksandmultiplicity}, this implies that $P_1$, $P_2$ and $P_3$ are in the same block. Since the multiplicity of $S_4$ in $P_j$ for $j=1,2,3$ are zero, $P_4$ has its own block. This shows there are exactly two blocks of $kA_5$. We can abstractly write central orthogonal idempotents $e_1$ and $e_2$ as
\begin{align*}
e_1&=\sum_{Af_i \cong P_1, P_2, \text{ or } P_3} f_i,\\
e_2&=\sum_{Af_i\cong P_4}f_i.
\end{align*}
Here, the principal block $B_0(kA_5)$ is $kA_5e_1$.
The unique (ring) decomposition of $kA_5$ is $kA_5\cong kA_5e_1\oplus kA_5e_2$
where $e_1=\sum_{Af_i \cong P_1, P_2, \text{ or } P_3} f_i$ and $e_2=\sum_{Af_i\cong P_4}f_i$.

\chapter{Conclusions}

We've seen all the projective indecomposable modules of $C_2\times C_2, A_4$ and $A_5$. We used the relations between those groups to find their projective indecomposable modules explicitly. To find projective indecomposable modules of $A_5$, it started with finding projective indecomposable modules of a cyclic $p$-group, the Klein four-group, the alternating 4-group and the alternating 5-group. 

A cyclic $p$-group algebra with its radical series is uniserial and each radical layer is isomorphic to the trivial module. The Klein four-group algebra has its radical series components all trivial as it is 2-group. To prove that every p-group has the trivial module as the only simple module, we used the weak form of Clifford's theorem. The alternating 4-group algebra has the Klein four-group as the largest normal $p$-subgroup of the algebra. By theorem, the Klein four-group acts trivially so the alternating four-group simple modules are isomorphic to the cyclic 3-group simple modules. By considering projective covers and injective envelopes, we were able to find $Q_i$. $A_5$ used the induced indecomposable modules in $A_4$ to find indecomposable modules of $A_5$. When finding the projective indecomposable $kA_5$-module $P_1$, we checked the structure of $P_1$ with the induced module on $Q_1$ to $A_5$.

We know blocks are good enough to study because any indecomposable module is contained in one block, so it's sufficient to study just one block that the module belongs to. How can we understand the usefulness of projective indecomposable modules? Is there a general projective indecomposable structure property, not confined to a group algebra $kG$? In the context of a regular module, what other dualities can be observed beyond those between radical and socle, or projective and injective?

\bibliography{main}{}
\bibliographystyle{plain}

\appendix
\chapter{Non-commutative Algebra}
 
\begin{Definition}\index{trivial representation}\label{trivial representation}
For any group $G$ and commutative ring $R$, the \textit{trivial representation} of $G$ is a morphism $\rho: G \rightarrow GL(R)$ defined by $\rho(g)=1_R$, for all $g\in G$. Here, $1_R$ is the identity map from $R$ to $R$. A \textit{trivial module}\index{trivial module} is a module $M$ over a ring $R$ where the ring element acts as the identity element for every element in the module. i.e., $g*a=a$ for all $g\in G$, and $m\in M$.
\end{Definition}

\begin{Remark}
The centre $Z(A)$ is a subalgebra of $A$.
\end{Remark}

\begin{Definition}
If $W$ is a linear subspace of a finite-dimensional vector space $V$, then the \textit{codimension}\index{codimension} of $W$ in $V$ is $\codim (W)=\dim(V)-\dim(W)$.
\end{Definition}

\begin{Definition}\label{composseries}
Let $M$ be a left $R$-module. A \textit{composition series}\index{composition series} of $M$ is a chain of submodules
\begin{equation*}
0=M_0< M_1<\cdots<M_{r-1}<M_r=M
\end{equation*}
such that $M_i/M_{i-1}$ is simple for $1\leq i\leq r$.
\end{Definition}

\begin{Theorem}[Jordan-H\"older Theorem]\label{JordHoldThm}\index{Jordan-H\"older Theorem}
Assume an $R$-module $M$ and two composition series $0=M_0<M_1<\cdots< M_{r-1}<M_r=M$ and $0=N_0<N_1<\cdots<N_{s-1}<N_s=M$. Then
\begin{enumerate}
\item
$r=s$.
\item
For some permutations $\sigma$ of $\{1,\cdots, r\}$, $M_i/M_{i-1} \cong N_{\sigma(i)}/N_{\sigma(i)-1}$.
\end{enumerate}
\end{Theorem}
\begin{proof}
Define $M_{ij}=M_{i-1}+(M_i \cap N_j)$ such that $M_{i-1}\leq M_{ij}\leq M_i$ and $N_{ij}=N_{i-1}+(N_i \cap M_j)$ such that $N_{i-1}\leq N_{ij}\leq N_i$. Then
$$M_{i-1}=M_{i0}\leq M_{i1}\leq \cdots \leq M_{i(s-1)}\leq M_{is}=M_i$$
for each $i=1,\cdots, r$.
$$N_{i-1}=N_{i0}\leq N_{i1}\leq \cdots \leq N_{i(s-1)}\leq N_{is}=N_i$$
for each $i=1,\cdots, s$.
Then $M_{ij}/M_{i(j-1)}\cong N_{ji}/N_{ji-1}$.
\end{proof}

\chapter{Representation Theory}
\begin{Lemma}\label{Vprojdecomp}
Let $V$ be a vector space and $p:V\rightarrow V$ be a linear map such that $p^2=p$. Then $V=\im (p)\oplus \ker(p)$.
\end{Lemma}
\begin{proof}
Let $x\in \im(p)\cap\ker(p)$. Since $x\in \im(p)$, there exists $y\in V$ such that $p(y)=x$. Then $p(x)=p^2(y)=p(y)=x=0$ since $x\in \ker(p)$. Let $x\in V$. Then $x=p(x)+(x-p(x))$, as $p(x-p(x))=p(x)-p^2(x)=p(x)-p(x)=0$.
\end{proof}

\begin{Theorem}[Maschke's Theorem]\label{Maschke'sthm}\index{Maschke's Theorem}
Let $k$ be a field where $\charr(k)$ does not divide $|G|$. Let $V$ be a $G$-space and $U$ be a $G$-subspace of $V$. Then there exists a $G$-subspace $W$ of $V$ such that $V=U\oplus W$.
Equivalently, the regular module $kG$ is semisimple.
\end{Theorem}
\begin{proof}
There exists a vector space $W'$ of $V$ such that $V=U\oplus W'$ as vector spaces. Let $p:V\rightarrow V$ be a projection onto $U$ with the kernel $W$. So $p^2=p$ and $\im(p)=U$. Define $q:V\rightarrow V$ by $q=1/|G|\sum_{x\in G}xpx^{-1}$. That is, $q(v)=1/|G|\sum_{x\in G}xp(x^{-1}v)=1/|G|\sum_{x\in G}(xp)(v)$. We show that this is a linear map. Let $g\in G$ and $v\in V$. 
$(gq)v=g(q(g^{-1}v))=g(1/|G|\sum_{x\in G}xp(x^{-1}g^{-1}v))=1/|G|\sum_{x\in G}(gx)p((gx)^{-1}v)\\=1/|G|\sum_{y\in G}yp(y^{-1}v)=q(v).$
So $gq=q$, and this shows $q\in\Hom_G(V,V)$. Let $v\in V$. Then $q(v)=1/|G|\sum_{x\in G}xp(x^{-1}v)$. Then $x^{-1}v\in V$, $p(x^{-1}v)\in U$ as $U=\im(p)$, $xp(x^{-1}v)\in U$ as $U$ is a $G$-space, so $q(v)\in U$. This shows $\im(q)\subseteq U$.
Let $u\in U$. Then $q(u)=1/|G|\sum_{x\in G}xp(x^{-1}u)=1/|G|\sum_{x\in G}x(x^{-1}u)=1/|G|\sum_{x\in G}u=1/|G|\cdot|G|u=u$. For all $y\in U$, $p(y)=y$ since $p$ is a projection onto $U$, so $p(x^{-1}u)x^{-1}u$. This shows $U\subseteq \im(q)$. It follows $U=\im(q)$.
Also if $v\in V$, then $q(v)\in \im(q)=U$, so $q(q(v))=q(v)$. This shows $q^2=q$.
By \ref{Vprojdecomp}, $V=U\oplus \ker(q)$.
\end{proof}

If $\charr(k)=p$ divides $|G|$, the function $q$ in the proof cannot be defined. Therefore, Maschke's Theorem impllies that $kG$ is semisimple if and only if $\charr(k)$ divides the order of $G$.

\begin{Definition}\index{G-set}
A \textit{G-set} is a set $X$ together with a map $G\times X \rightarrow X$ defined by $(g,x) \mapsto gx$ such that $e x=x$ for all $x\in X$ and $(gh)x=g(hx)$ for all $g, h \in G$ and $x\in X$. Here, $e$ is the identity element of $G$.
\end{Definition}

\begin{Definition}\index{stabiliser}
If $x\in X$, $\Stab_G(x)=\{g\in G : gx=x\} \leq G$ is called the \textit{stabilizer} of $x$ under $G$.
\end{Definition}

\begin{Definition}
Let $x\in X$. The \textit{orbit}\index{orbit} of $x$ under $G$ is $x^G=\{gx: g\in G\}$.
\end{Definition}

\begin{Definition}
We say that $X$ is \textit{transitive}\index{transitive} if there exists only one orbit on $X$. In other words, for all $x,y\in X$, there exists $g\in G$ such that $y=gx$.
\end{Definition}

\begin{Proposition}\label{transitiveXisoG/H}
Suppose $X$ is a transitive $G$-set. Let $x\in X$. Then $X \cong G/H$ where $H=\Stab_G(x)$.
\end{Proposition}

\begin{proof}
Define $f: G/H \rightarrow X$ by $f(gH)=gx$. Then the map is well-defined. If $g_1 H= g_2 H$, then there exists $h\in H=\Stab_G(x)$ such that $g_2=g_1h$. $f(g_2H)=g_2x=g_1hx=g_1x=f(g_1H)$. The proof that shows $f$ is a $G$-set map, is onto and one-to-one is omitted.
\end{proof}

\begin{Corollary}\label{transtiveGsetordequalasordGoverordStabx}
Let $X$ be a transitive $G$-set and $x\in X$. Then $$|X|=\frac{|G|}{|\Stab_G(x)|}.$$
\end{Corollary}
\begin{proof}
By \ref{transitiveXisoG/H}, $X \cong G/\Stab_G(x)$, so $|X|=[G:\Stab_G(x)]=|G|/|\Stab_G(x)|$.
\end{proof}

For any $G$-set $X$ and $x\in X$, the orbit $x^G$ is a $G$-set.

\begin{Theorem}[Orbit Stabilizer Theorem]
Let $x\in X$ where $X$ is a $G$-set. Then
$$|x^G|=\frac{|G|}{|\Stab_G(x)|.}$$
\end{Theorem}
\begin{proof}
The theorem follows from \ref{transtiveGsetordequalasordGoverordStabx}.
\end{proof}

\begin{Example}\label{GasaGsetunderconjugation}
$G$ is also a $G$-set under conjugation. We can set a $G$-set $X$ as $X=G$ and the action $g\in G$ on $x\in G$, $g*x$, as $gxg^{-1}$. Then the orbits are the conjugacy classes of $G$, $\Stab_G(g)=C_G(g)$.
\end{Example}

\begin{Corollary}\label{ordofconjugacyclassesofGisordofG/centerofg}
$|g^G|= [G: C_G(g)]$.
\end{Corollary}

\begin{proof}
The corollary follows from the orbit stabiliser theorem and \ref{GasaGsetunderconjugation}.
\end{proof}

\section{Character Table of $C_3$}\label{chartabofC3}
\begin{table}[h]
\centering
\begin{tabular}{ c | c  c c  }
\hline
 \hline
  $g$   & $1$ & $g$ & $g^2$\\
   $|C_G(g)|$  & $3$ & $3$ & $3$\\
 \hline
 $\chi_1$   & 1    &1&   1\\
$\chi_{e^{\frac{\pi i}{3}}}$&   1  & $e^{\frac{\pi i}{3}}$   &$e^{\frac{2 \pi i}{3}}$\\
$\chi_{e^{\frac{2\pi i}{3}}}$&1 & $e^{\frac{2\pi i}{3}}$&  $e^{\frac{\pi i}{3}}$\\
 \hline
 \hline
\end{tabular}
\caption{The character table of $C_3$}
\end{table}

\newpage
\printindex

\end{document}